\documentclass[10pt,reqno]{amsart}
\usepackage{amssymb}
\usepackage{amsmath}
\usepackage{amsthm}
\usepackage{mathtools,mathrsfs}
\usepackage{graphicx}
\usepackage{cite} 

\usepackage{subfigure}

\usepackage{caption}

\usepackage{bm}

\usepackage{enumitem}

\usepackage[colorlinks=true,urlcolor=blue,
citecolor=red,linkcolor=blue,linktocpage,pdfpagelabels,
bookmarksnumbered,bookmarksopen]{hyperref}%per colore-link

\usepackage{ifthen} % pacchetto necessario per un minimo di
\provideboolean{shownotes} % dichiariamo una variabile booleana
\setboolean{shownotes}{true} % la definiamo come "vera" o "falsa"
\newcommand{\margnote}[1]{
\ifthenelse{\boolean{shownotes}}%
{\marginpar{\raggedright\tiny\texttt{#1}}}%
{}%
}

\newcommand{\hole}[1]{
\ifthenelse{\boolean{shownotes}}%
{\begin{center} \fbox{ \rule {.25cm}{0cm}
\rule[-.1cm]{0cm}{.4cm} \parbox{.85\textwidth}{\begin{center}
\texttt{#1}\end{center}} \rule {.25cm}{0cm}}\end{center}}
{}
}

\usepackage[titletoc,title]{appendix}

\numberwithin{equation}{section} %riparte da zero ogni sezione

\newcounter{cont}[section]

\newtheorem{theorem}[cont]{Theorem}
\newtheorem{proposition}[cont]{Proposition}

\newtheorem{lemma}[cont]{Lemma}

\theoremstyle{definition}
\newtheorem{remark}[cont]{Remark}
\newtheorem{definition}[cont]{Definition}

 \theoremstyle{remark}

\renewcommand{\k}{\kappa}

\renewcommand{\Re}{\mathrm{Re}\,} 
\renewcommand{\Im}{\mathrm{Im}\,}

\newcommand{\e}{\varepsilon}

\newcommand{\N}{\mathbb{N}}
\newcommand{\R}{\mathbb{R}}
\newcommand{\C}{\mathbb{C}}

\newcommand{\E}{\mathbb{E}}

\newcommand{\tiu}{\widetilde{u}}
\newcommand{\tiv}{\widetilde{v}}

\newcommand{\cL}{{\mathcal{L}}}

\newcommand{\cD}{{\mathcal{D}}}

\newcommand{\cR}{{\mathcal{R}}}
\newcommand{\tcL}{\widetilde{\mathcal{L}}}

\newcommand{\ccC}{\mathscr{C}}
\newcommand{\ccB}{\mathscr{B}}

\newcommand{\ess}{\sigma_\mathrm{\tiny{ess}}}
\newcommand{\ptsp}{\sigma_\mathrm{\tiny{pt}}}

\begin{document}

\title[Existence and stability of viscous-dispersive shock profiles]{Existence and spectral stability analysis of viscous-dispersive shock profiles for isentropic compressible fluids of Korteweg type}

\author[R. Folino]{Raffaele Folino}

\address[R. Folino]{Departamento de Matem\'aticas y Mec\'anica\\Instituto de 
Investigaciones en Matem\'aticas Aplicadas y en Sistemas\\Universidad Nacional Aut\'onoma de 
M\'exico\\Circuito Escolar s/n, Ciudad Universitaria C.P. 04510 Cd. Mx. (Mexico)}

\email{folino@aries.iimas.unam.mx}

\author[C. Lattanzio]{Corrado Lattanzio}

\address{{\rm (C. Lattanzio)} Dipartimento di Ingegneria e Scienze dell'Informazione e Matematica\\Universit\`a degli Studi dell'Aquila\\via Vetoio (snc), Coppito I-67010, L'Aquila (Italy)}

\email{corrado.lattanzio@univaq.it}

\author[R. G. Plaza]{Ram\'on G. Plaza}

\address[R. G. Plaza]{Departamento de Matem\'aticas y Mec\'anica\\Instituto de 
Investigaciones en Matem\'aticas Aplicadas y en Sistemas\\Universidad Nacional Aut\'onoma de 
M\'exico\\Circuito Escolar s/n, Ciudad Universitaria C.P. 04510 Cd. Mx. (Mexico)}

\email{plaza@aries.iimas.unam.mx}



% 35Q30, 35Q35

%%%%%%%%%%%%%%%%%%%%%%%%%%%%%%%\nocite{*}

\begin{abstract} 
The system describing the dynamics of a compressible isentropic fluid exhibiting viscosity and internal capillarity in one space dimension and in Lagrangian coordinates, is considered. It is assumed that the viscosity and the capillarity coefficients are nonlinear smooth, positive functions of the specific volume, making the system the most general case possible. It is shown, under very general circumstances, that the system admits traveling wave solutions connecting two constant states and traveling with a certain speed that satisfy the classical Rankine--Hugoniot and Lax entropy conditions, and hence called viscous-dispersive shock profiles. These traveling wave solutions are unique up to translations and have arbitrary amplitude. The spectral stability of such viscous-dispersive profiles is also considered. It is shown that the essential spectrum of the linearized operator around the profile (posed on an appropriate energy space) is stable, independently of the shock strength. With the aid of energy estimates, it is also proved that the point spectrum is also stable, provided that the shock amplitude is sufficiently small and a structural condition on the inviscid shock is fulfilled. 
\end{abstract}

%%%%%%%%%%%%%%%%%%%%%%%%%%%\tableofcontents

\maketitle

\section{Introduction}
\label{secintro}

In this paper we study the system of equations describing the dynamics of a compressible isentropic fluid exhibiting viscosity and internal capillarity in a one dimensional unbounded domain and in Lagrangian coordinates,
\begin{equation}
\label{eq:NSK-Lag}
\left\{
\begin{aligned}
v_t -  u_x &= 0,\\
u_t + p(v)_x &= \displaystyle{\Big(\frac{\mu(v)}v u_x\Big)_x - \Big(\k(v)v_{xx}+\tfrac{1}{2}\k'(v)v^2_x\Big)_x}.
\end{aligned}
\right.
\end{equation} 
Here $x\in \mathbb{R}$ and $t>0$ denote the space and time variables, respectively. The scalar unknowns $u=u(x,t)$ and $v=v(x,t)$ denote the velocity and the specific volume of the fluid. The function $p = p(v)$ is the pressure and  $\mu = \mu(v)$ and $\kappa = \kappa(v)$ are smooth positive functions of the specific volume, representing generic nonlinear viscosity and capillarity coefficients, respectively.

Motivated by an early work by J. D. van der Waals \cite{vdW1894} and in order to describe internal capillarity effects in diffuse interfaces for liquid vapor flows, in 1901 the Dutch physicist D. J. Korteweg \cite{Kortw1901} proposed constitutive equations for stress tensors that included density gradients and which were, in general, incompatible with the Clausius--Dulhem inequality of equilibrium thermodynamics. This problem was later circumvented by Dunn and Serrin \cite{DS85}, who rigorously rederived a system of equations to account for compressible fluids endowed with internal capillarity under the framework of Rational Mechanics. System \eqref{eq:NSK-Lag} is the isentropic version of the model derived by Dunn and Serrin in one space dimension, in the case of generic nonlinear viscosity and capillarity coefficients, and in Lagrangian coordinates (for a derivation of system \eqref{eq:NSK-Lag} from the original Dunn and Serrin system in Eulerian coordinates, see \cite{PlV22}). Notice that when the capillarity coefficient is equal to zero, $\kappa(v) \equiv 0$ for all $v > 0$, then system \eqref{eq:NSK-Lag} reduces to the compressible Navier--Stokes equations. For this reason, \eqref{eq:NSK-Lag} is also known as the one-dimensional \emph{Navier--Stokes--Korteweg (NSK) system} (cf. \cite{ChZha14,CLS19,BrGL-V19}).

The NSK system has been the subject of numerous mathematical results in the scientific literature, pertaining to the local and global existence of weak \cite{BrDjL03,Hasp09,Hasp11,DD01}, strong \cite{Kot08,Kot10} and classical \cite{HaLi94,HaLi96a,HaLi96b} solutions, as well as to the study of phase transitions \cite{Sl83,Sl84a,FrKo17,FrKo19,HaSl83}, or the decay of perturbations of thermodynamical equilibria \cite{WaTa11,TZh14,GLZh20,PlV22,GLT17}. A particularly important field of study concerns the emergence and the stability of \emph{viscous-dispersive shock profiles}, which are traveling wave solutions to system \eqref{eq:NSK-Lag} of the form $(v,u)(x,t) = (V,U)(x-st)$, traveling with speed $s \in \R$ and connecting two constant states $(V^\pm, U^\pm)$ as $x -st \to \pm \infty$. It is assumed that the triplet $(V^\pm, U^\pm,s)$ constitutes a shock front solution to the system in the absence of viscosity and capillarity effects, and hence satisfying the classical Rankine--Hugoniot jump conditions as well as Lax entropy conditions. The interplay between diffusion (viscosity) and dispersion (capillarity) effects has played a preponderant role in the literature.

The existence theory of viscous-dispersive shock profiles dates back to early works by Smoller and Shapiro \cite{SmSh82}, Pego \cite{Pe85} and Khodja \cite{KhdPhD89}, mainly for constant viscosity and capillarity coefficients. More recent existence results include \cite{ChZha14,HKKL25,Hoe14,HoAb07,LaZ21a} for the NSK system or for related models, such as Quantum Hydrodynamics (QHD) systems or scalar equations. The theory of stability of viscous-dispersive shock profiles  is much less developed. The only known stability results pertain to scalar equations \cite{PaW04,HZ00}, to isentropic fluid dynamics with constant capillarity and constant viscosity \cite{KhdPhD89,ZLY16,Z00,HKKL25,Sl83,Sl84a}, or with constant capillarity and variable viscosity \cite{Hu09,CHZ15}.

In this paper, we study both existence and stability in the most general case: where the viscosity and capillarity coefficients, treated as thermodynamic potentials, are nonlinear functions of the specific volume $v > 0$. Our first result is the most general possible existence theorem for isentropic NSK systems. It guarantees that a unique (up to translations) viscous-dispersive shock profile exists for any $(V^\pm, U^\pm,s)$ satisfying the usual jump and entropy conditions, \emph{regardless of the shock amplitude and for any smooth positive viscosity and capillarity functions} (see Theorems \ref{theo:ex-L} and \ref{theo:ex-2-L} below). The hypotheses are minimal and hence the conclusion applies to many different situations. The analysis is based on a careful description of the stable and unstable manifolds around equilibria and on the introduction of a novel auxiliary system which results from removing the viscous part. Regarding the stability analysis, our work establishes the \emph{spectral} stability property, that is, focuses on the \emph{linearized} operator around the profile posed on an appropriate energy space and aims to establish the absence of spectra with positive real part. It is well-known from the theory of purely viscous shocks (see, e.g., \cite{HuTh02}), that spectral stability is tantamount to linear stability under zero mass perturbations. In this context, we prove that the essential spectrum of the linearized operator is stable under very general circumstances (see Theorem \ref{thmstabess} below). The stability of the point spectrum, however, is more delicate. We were able to prove the stability of point eigenvalues only in the case of sufficiently weak shocks that satisfy, in addition, a particular condition (see equation \eqref{eq:imp-ass} below) which holds trivially in the case of constant capillarity, hence extending previous results. This result is based on energy estimates at the spectral level, extending the results of Humpherys \cite{Hu09} for constant viscosity and capillarity.  Our existence and stability results stand out because they apply to the more general case of variable coefficients, which is new in the literature. A special mention deserves the recent analysis by Han \emph{et al.} \cite{HKKL25}, which establishes the \emph{nonlinear} stability of sufficiently weak profiles for constant viscosity and capillarity, without assuming the zero-mass condition. The authors employ the so called $a$-method in order to get rid of the zero-mass assumption. 
%
%\hole{(B) TO DO: Finish intro... Cite \cite{CHZ15,HKKL25,BMT23,SmSh82}}

\subsection*{Plan of the paper}

Section \ref{sec:prel-L} contains the description of the shock front solutions to the ``inviscid" version of system \eqref{eq:NSK-Lag} known as the $p$-system (that is, in the absence of viscosity and capillarity) as well as the resulting ODE system for the profiles when the capillarity and viscosity are swichted on. The central Section \ref{sec:ex-L} is devoted to prove the existence of viscous-dispersive shock profiles to system \eqref{eq:NSK-Lag} in the most general case with nonlinear viscosity and capillarity coefficients. Section \ref{sec:properties-small-profiles-L} contains a description of further properties of such profiles, specially in the case of small-amplitude shocks. In Section \ref{sec:essential-L} the linearization around a viscous-dispersive profile is established and the stability of its essential spectrum is proved under very general circumstances.  Section \ref{sec:point-L} is devoted to proving that the point spectrum is also stable, but under more restrictive conditions which include sufficiently small shock amplitudes. 
Finally, we collect the results valid for the system in Eulerian coordinates in the appendix.

\subsection*{Notations} We denote the real and imaginary parts of a complex number $\lambda \in \C$ by $\Re\lambda$ and $\Im\lambda$, respectively, as well as complex conjugation by ${\lambda}^*$. Standard Lebesgue and Sobolev spaces of complex-valued functions on the real line, namely $L^2(\R;\C)$ and $H^m(\R;\C)$ with $m \in \N$, will be denoted as $L^2$ and $H^m$, respectively. They are endowed with the standard inner products and norms. Linear operators acting on infinite-dimensional spaces are indicated with calligraphic letters (e.g., $\cL$). If $X$ and $Y$ are Banach spaces then $\ccC(X,Y)$ and $\ccB(X,Y)$ will denote the spaces of all closed and bounded linear operators from $X$ to $Y$, respectively. For any $\cL \in \ccC(X,Y)$ we denote its domain as $\cD(\cL) \subseteq X$ and its range as $\cR(\cL) = \cL (\cD(\cL)) \subseteq Y$. 

\section{Preliminaries}
\label{sec:prel-L}

In this Section
we recall   some basic definitions and properties of the underlying 
equation for gas dynamics, that is, the well-known $p$--system,
\begin{equation}\label{eq:Euler-Lag}
\begin{cases}
v_t -  u_x = 0,\\
u_t + p(v)_x = 0.
\end{cases}
\end{equation}
In \eqref{eq:Euler-Lag}, we recall that $v=v(x,t) > 0$ is the specific volume, (equal to $1/\rho$, $\rho$ being the density),  $u=u(x,t)$ is the velocity, and $\mu(v)$ and $\k(v)$ are the (positive) viscosity and capillarity  coefficients, both depending on the specific volume, which are assumed to be sufficiently smooth.
As it is customary, for the sufficiently smooth pressure function $p(v)$,  we assume 
\begin{equation}
\label{cond_pressure-L}
%p \in C^2(\Rp, \Rp)\mbox{, }
p'(v) < 0, \quad p''(v) > 0, \qquad \mbox{ for }v > 0.
\end{equation}
Finally, it is worth recalling the following relations lining the above quantities with the corresponding ones in Eulerian coordinates (for details, see \cite{PlV22}):
\begin{align}\label{eq:relationlagrangian}
    p(v) = \tilde{p}(1/v), & & \mu(v) = \tilde{\mu}(1/v), & & \k(v)=\frac{\tilde{\k}(1/v)}{v^5}
\end{align}
(here the tilded variables denote the corresponding thermodynamic potential functions of the density, $\rho > 0$, in Eulerian coordinates). As it is well-known,  system \eqref{eq:Euler-Lag}  
 can be recast in the conservative form 
$$W_t + F(W)_x = 0,$$ 
for the vector $W = (v, u)^\top$ with  flux $F(W) = F(v,u) =  (-u,  p(v))^\top$.
Moreover, the Jacobian of $F$ is given by
\begin{equation*}
DF(v,u) = 
\begin{pmatrix}
0 & -1\\
  p'(v) & 0
\end{pmatrix},
\end{equation*}
and the  $p$--system is hyperbolic, with characteristic speeds (eigenvalues of $DF(v,u)$) given by 
\begin{equation*}
\lambda_1(W) =  -  \sqrt{-p'(v)},\qquad \lambda_2(W) = \sqrt{-p'(v)}.
\end{equation*}

Let us now consider the full system \eqref{eq:NSK-Lag} with nonlinear (generic) viscosity and capillarity coefficients. Viscous-dispersive shock profiles are traveling waves solutions to \eqref{eq:NSK-Lag} of the form
\begin{equation*}
    v(x,t) = V(x - st), \qquad \qquad u(x,t) = U(x - st), 
\end{equation*}
where $s \in \R$ is the speed of the traveling wave, with prescribed end states at $\pm\infty$: 
\begin{equation}\label{eq:profiles-end-Lag}
    V^\pm = \lim_{y \rightarrow \pm \infty} V(y), \qquad\qquad   U^\pm = \lim_{y \rightarrow \pm \infty} U(y),
\end{equation}
where we introduced the parameter along the profile $y = x - s t$. 
Clearly, the profiles $V$, $U$ solve
\begin{equation}\label{eq:profiles-L}
\begin{cases}
    -s V' - U' = 0,\\
\displaystyle{-s U' +  p(V)' = \left ( \frac{\mu(V)}{V}U' \right )' -   \left (\k(V)V'' + \frac12 \k'(V)(V')^2 \right )'}.
\end{cases}
\end{equation}
On the other hand, the triplets $(s;V^\pm,U^\pm)$ shall verify appropriate conditions to define admissible shocks for the $p$--system  as recalled here below.
\subsection{Rankine--Hugoniot conditions and the equations for the profiles}\label{subsec:RH-L}
The first requirement linking the speed $s$ and the end states are the Rankine--Hugoniot conditions for the system \eqref{eq:Euler-Lag}, that is
\begin{align}
&s (V^+ - V^-) = U^- - U^+, \label{RH1-L}\\
&s (U^+ - U^-) = p(V^+) - p(V^-). \label{RH2-L}
\end{align}
Relations \eqref{RH1-L} and \eqref{RH2-L} also imply that
\begin{equation}
\label{eq:RH-rewr-L}
    s^2(V^+ - V^-) = p(V^-) - p(V^+),
\end{equation}
and, consequently, the following constants are well--defined:
\begin{equation} \label{def_C-L}
    \begin{aligned}
        A&:= sV^++U^+=sV^-+U^-,\\
        B&:= sU^+ - p(V^+)=sU^- - p(V^-),\\
        C&:= s^2V^+ + p(V^+)=s^2V^- + p(V^-).
    \end{aligned}
\end{equation}
Hence, integrating the equations in \eqref{eq:profiles-L}, we readily obtain 
\begin{align*}
    &U(y)=-sV(y)+A;\\
    &\frac{\mu(V)}{V}U'   -    \k(V)V'' - \frac12 \k'(V)(V')^2 
    = p(V) - sU + B.
\end{align*}
%\label{eq:J-R}
Therefore, the profile $V$ of the specific volume solves  the following second order ODE
%By substituting \eqref{eq:J-R} into the  equation above, we end up with
\begin{equation*}
  -s\frac{\mu(V)}{V}V' -   \k(V)V'' - \frac12 \k'(V)(V')^2 = p(V) + s^2 V -sA + B = 
    p(V) + s^2 V -C, 
\end{equation*}
which can be rewritten as the $2\times2$ dynamical system 
\begin{equation}\label{eq:V-Q}
    \begin{cases}
        V' = Q,\\
        Q' = \displaystyle{-\frac{ 1}{\k(V)} \left [ p(V) + s^2 V - C 
         +s\frac{\mu(V)}{V}Q
         + \frac{1}{2}  \k'(V) Q^2\right ]}.
    \end{cases}     
\end{equation}
% where 
% \begin{equation}\label{eq:f(R)}
% \begin{aligned}
%      f(R)&:=p(R)+\frac{(sR-A)^2}{R}-s^2R+sA-B=p(R)+\frac{A^2}R-sA-B \\
%      & = p(R) - p(R^+) +\frac{A^2}{R} - \frac{A^2}{R^+}  \\
%      & = p(R) - p(R^-) +\frac{A^2}{R} - \frac{A^2}{R^-}.
% \end{aligned}
% \end{equation}
%We shall study existence of solutions to \eqref{eq:R-Q} connecting $(R^-,0)$ and $(R^+,0)$ in Section \ref{sec:ex}.
%
% We conclude this Section by eliminating the variables $U^\pm$, referring to the velocity,  also in relations \eqref{RH1-L}--\eqref{RH2-L}. For this,  
% from the first equation  it follows that 
% $U^+ = U^- - s (V^+ - V^-)$ and then \eqref{RH2-L} becomes
% \begin{equation}\label{quadratic_eq-L}
%    U^- = U^+ - \frac{1}{s}\left [ p(V^+) - p(V^-)\right ]
% \end{equation}
% The discriminant of this  quadratic equation for $J^-$ is given by
% \begin{equation*}
%     D:=4 R^+ R^- \left(\frac{p(R^-) - p(R^+)}{R^- - R^+}\right).
% \end{equation*}
% Since the assumption \eqref{cond_pressure} requires that $p(\rho)$ is strictly increasing for $\rho > 0$, we readily conclude  $D > 0$ and  equation \eqref{quadratic_eq} has two solutions
% \begin{equation*}
% J^{-}_{1,2} = s R^- \pm d,
% \end{equation*}
% where
% \begin{equation}\label{eq:d}
%     d := \sqrt{R^- R^+} \sqrt{\frac{p(R^-) - p(R^+)}{R^- - R^+}} > 0.
% \end{equation}

In the remaining part of this preliminary Section we shall introduce also the admissibility conditions for the discontinuity under consideration. 
% In particular, we shall determine the unique correct solution $J^{-}$ of \eqref{quadratic_eq}, according to the choice  of a shock of the first or second family.
%
%
%
%We will consider discontinuities $(s, U^{\pm})$ which define Lax shocks, that is
%\begin{equation*}
%\lambda_k(U^+) < s < \lambda_k(U^-),\mbox{ }k = 1,\mbox{ }2.
%\end{equation*}
%
\subsection{Compressive Lax shocks}\label{subsec:Lax-L}
The discontinuity $(s; W^{\pm}) = (s;V^\pm,U^\pm)$ is a weak solution of the $p$--system \eqref{eq:Euler-Lag} in view of the Rankine--Hugoniot conditions outlined in the previous Section. Now, we shall require it defines a compressible Lax shock of the first or the second family, as detailed here below \cite{Da4e}.
\subsubsection{Compressive Lax 1--shock (backward shock)}
Suppose the end states $W^{\pm}$ and the speed $s$ satisfy the condition for a compressive Lax 1--shock, that is
\begin{equation}\label{cond_1_shock-L}
\lambda_1(W^+) < s < \lambda_1(W^-), \qquad \ s< \lambda_2(W^+).
\end{equation}
In particular we have $s<0$ (and hence the name \emph{backward shock}) and the second condition is trivially verified being $\lambda_2(W)>0$. 
Moreover, \eqref{cond_1_shock-L} also implies 
\begin{equation*}
    -\sqrt{-p'(V^-)} = \lambda_1(W^-) > \lambda_1(W^+) %= \frac{J^+}{R^+} - \frac{J^-}{R^-}-(c_s(R^+) - c_s(R^-)) 
    =  -\sqrt{-p'(V^+)},
\end{equation*}
that is 
\begin{equation}\label{eq:V->V+1shock-L}
    V^- > V^+
\end{equation}
in view of the monotonicity of $p'$ stated in \eqref{cond_pressure-L}.
Finally, from the above relation we conclude
\begin{equation*}
    U^+ - U^- = s(V^--V^+) < 0.
\end{equation*}
\subsubsection{Compressive Lax 2--shock (forward shock)}
On the other hand, suppose the end states $W^{\pm}$ and the speed $s$ satisfy the condition for a compressive Lax 2--shock:
\begin{equation}
\label{cond_2_shock-L}
\lambda_2(W^+) < s < \lambda_2(W^-), \qquad  \ s> \lambda_1(W^-).
\end{equation}
Then $s>0$ (\emph{forward shock}) and 
 \eqref{cond_2_shock-L} this time rewrites
\begin{equation*}
    \sqrt{-p'(V^-)} = \lambda_2(W^-) >  \lambda_2(W^+)  
    = \sqrt{-p'(V^+)},
\end{equation*}
which implies
\begin{equation}\label{eq:V-<V+2shock-L}
    V^+> V^-
\end{equation}
and again
\begin{equation*}
    U^+ - U^- = s(V^--V^+) < 0.
\end{equation*}

\section{Existence of viscous-dispersive shock profiles}
\label{sec:ex-L}

In this Section we prove existence of a solution $(V,Q)$ to \eqref{eq:V-Q} connecting $(V^-,0)$ with $(V^+,0)$, provided that the end states $(V^\pm,U^\pm)$ and the speed $s$ define a compressive $1$--shock (backward shock) as recalled here above. The proof in the case of a forward shock with $s > 0$ is analogous and we omit it. Hence, in the case at hand we recall that, in particular (see \eqref{eq:V->V+1shock-L}),
\begin{equation*}
    s<0, \qquad V^->V^+, \qquad U^+ - U^- = s(V^--V^+) < 0.
\end{equation*}

As it will be soon manifest, in order to prove the existence of the profile, a crucial role will be played by the following \emph{auxiliary system},
\begin{equation}\label{eq:V-Q-reduced}
    \begin{cases}
        V' = Q,\\
        Q' = \displaystyle{-\frac{f(V)}{\k(V)} - \frac{1}{2} \frac{\k'(V)}{\k(V)}  Q^2}
    \end{cases}     
\end{equation}
obtained from \eqref{eq:V-Q} after removing the ``viscous'' term $- (s \mu(V)/V \k(V)) Q$, and where
\begin{equation}\label{eq:def-f-Lag}
    f(V) : =  p(V) + s^2V-C.
\end{equation}
Since $f(V^\pm)=0$ (see \eqref{def_C-L}), we conclude that $(V^\pm,0)$ are stationary points of both \eqref{eq:V-Q} and \eqref{eq:V-Q-reduced}. Hence,
let us start by analyzing the local behavior of these systems around them.

\subsection{Linearization of the auxiliary system at equilibrium points}
\label{subsec:DynSystLinear-L}

For $\Phi=(V,Q)^\top$, the auxiliary system \eqref{eq:V-Q-reduced} can be rewritten as follows:  
\begin{equation*}
\Phi'=G(\Phi):=\begin{pmatrix}
    Q \\
    \displaystyle{-\frac{f(V)}{\k(V)}  - \frac{1}{2}\frac{\k'(V)}{\k(V)}Q^2}
\end{pmatrix}.
\end{equation*} 
Recalling that $f(V^\pm)=0$, the Jacobian of $G$, evaluated at the equilibrium points $(V^\pm,0)$, is given by
\begin{equation*}
D_\Phi G(V^\pm,0)=\begin{pmatrix}
    0 & 1 \\
    \displaystyle{-\frac{f'(V^\pm)}{\k(V^\pm)}} & 0
\end{pmatrix},
\end{equation*}
where 
\begin{equation*}
    f'(V) = p'(V) + s^2.
\end{equation*}
Therefore, the eigenvalues $\Lambda = \Lambda(V^\pm)$ of this matrix are solution of the characteristic equation
\begin{equation*}
\Lambda^2+\frac{f'(V^\pm)}{\k(V^\pm)}=0.
\end{equation*}
% Thus, there are two possibilities:
% \begin{itemize}
%     \item if $f'(R^\pm)>0$ then $D_\Phi G(R^\pm,0)$ has two real eigenvalues and $(R^\pm,0)$ is a saddle point for \eqref{eq:R-Q-reduced},
%     \item if $f'(R^\pm)<0$, then $D_\Phi G(R^\pm,0)$ has two purely imaginary eigenvalues and $(R^\pm,0)$ is a focus for \eqref{eq:R-Q-reduced}.
% \end{itemize}
%In order to establish the sign of $f'(R^\pm)$, let us use the definition of $B$ \eqref{def_B} to rewrite  as
%\begin{equation}\label{eq:def-f}
 %   f(R)=p(R)-p(R^+)+\frac{A^2}R-\frac{A^2}{R^+}=
  %  p(R)-p(R^-)+\frac{A^2}R-\frac{A^2}{R^-}.
%\end{equation} 
Due to the convexity of $p$ (see assumption \eqref{cond_pressure-L}), the function $f$ is convex for any $V>0$, and therefore
\begin{equation}\label{eq:sign-f-L}
    f(V)<0,\ \mbox{for any}\ V\in(V^+,V^-);\qquad f'(V^+)<0, \ f'(V^-)>0,
\end{equation}
 the latter properties being also a direct consequence of \eqref{cond_1_shock-L}.
Hence, the last inequalities in \eqref{eq:sign-f-L} imply that $(V^-,0)$ is a centre and  $(V^+
,0)$ is a saddle point for
\eqref{eq:V-Q-reduced}.
Finally, 
for future purposes, we compute the eigenvectors of $D_\Phi G(V^+,0)$, which give the direction of the stable/unstable manifolds corresponding to the saddle point $(V^+,0)$. 
For this purpose, let us denote  
\begin{equation*}
    \Lambda_s(V^+):=-\sqrt{-\frac{f'(V^+)}{\k(V^+)}}; \qquad \Lambda_u(V^+):=\sqrt{-\frac{f'(V^+)}{\k(V^+)}}.
\end{equation*}
The stable manifold of $(V^+,0)$ is tangent to the eigenspace corresponding to the negative eigenvalue of $D_\Phi G(V^+,0)$, that is $\lambda_s(V^+)$, while the unstable manifold is tangent to the eigenspace corresponding to the positive eigenvalue $\lambda_u(V^+)$. These eigenspaces are  given by the following eigenvectors  of $D_\Phi G(V^+,0)$,
\begin{equation}\label{eq:tang-lin-L}
    \tau_s(V^+):=(1,\Lambda_s(V^+)), \qquad \tau_u(V^+):=(1,\Lambda_u(V^+)),
\end{equation} 
 corresponding to the negative eigenvalue $\Lambda_s(V^+)$ and the positive one $\Lambda_u(V^+)$, respectively.

\subsection{Linearization of the full system at equilibrium points}
\label{subseclinfull}

We now analyze the behavior of the linearization of \eqref{eq:V-Q} at the equilibrium points $(V^\pm,0)$.
For this system, the Jacobian evaluated at $(V^\pm,0)$ is given by
\begin{equation*}
    J(V^\pm):=\begin{pmatrix}
    0 & 1 \\
    \displaystyle{-\frac{f'(V^\pm)}{\k(V^\pm)}} & -\displaystyle\frac{s\mu(V^\pm)}{\k(V^\pm)V^\pm}
\end{pmatrix}.
\end{equation*}
Then, the  eigenvalues $\lambda = \lambda(V^\pm)$ of the above matrix satisfy the characteristic equation
\begin{equation*}
    \lambda^2+\displaystyle\frac{s\mu(V^\pm)}{\k(V^\pm)V^\pm}\lambda+\frac{f'(V^\pm)}{\k(V^\pm)}=0,
\end{equation*}
with discriminant given by 
\begin{equation}\label{eq:discriminant-L}
  \Delta(V^\pm)=\left[\frac{s\mu(V^\pm)}{\k(V^\pm)V^\pm}\right]^2-\frac{4f'(V^\pm)}{\k(V^\pm)}.
\end{equation}
Since $f'(V^+)<0$, we conclude $\Delta(V^+)>0$ and $(V^+,0)$ is a saddle point for system \eqref{eq:V-Q}. Moreover,  the (real) eigenvalues of $J(V^+)$ are 
\begin{align*}
    \lambda_s(V^+)&:=-\frac{s\mu(V^+)}{2\k(V^+)V^+}-\frac{\sqrt{\Delta(V^+)}}2<0,\\
    \lambda_u(V^+)&:=-\frac{s\mu(V^+)}{2\k(V^+)V^+}+\frac{\sqrt{\Delta(V^+)}}2>0.
\end{align*}

On the other hand, $f'(V^-)>0$ implies that $(V^-,0)$ is an unstable equilibrium for \eqref{eq:V-Q}, because both eigenvalues of $J(V^-)$ have negative real part. Moreover, this point  is either an unstable node or an unstable focus, depending on whether $\Delta(V^-)$ is positive or negative, respectively. Among other values, the sign of $\Delta(V^-)$ depends in particular on the relative magnitudes of the viscosity and the capillarity coefficients  $\mu(V^-)$ and $\k(V^-)$ and it is for sure negative for $\mu(V^-)\ll 1$. Let us define
\begin{equation}
\label{defofeta}
\eta(v) := \frac{\mu(v)}{\sqrt{\kappa(v)}} > 0, \qquad \text{for all} \; v > 0.
\end{equation}
This function of the specific volume measures the interplay between diffusion (viscosity) and dispersion (capillarity) effects. Then it is clear that $\Delta(V^-) < 0$ if and only if
\[
\eta(V^-) < \frac{2 V^-\sqrt{f'(V^-)}}{|s|}.
\]
Similarly to \eqref{eq:tang-lin-L}, 
\begin{equation}\label{eq:tang-tilde-L}
    \nu_s(V^+):=(1,\lambda_s(V^+)), \qquad \nu_u(V^+):=(1,\lambda_u(V^+)),
\end{equation}
are the  two eigenvectors of $J(V^+)$ corresponding to the negative (resp.\ positive) eigenvalue $\lambda_s(V^+)$ (resp.\ $\lambda_u(V^+)$) of that matrix.

Finally, notice that 
\begin{equation*}
    \sqrt{\frac{-f'(V^+)}{\k(V^+)}} < \frac{\sqrt{\Delta(V^+)}}2< \left |\frac{s\mu(V^+)}{2\k(V^+)V^+}\right | +\sqrt{\frac{-f'(V^+)}{\k(V^+)}}< \left |\frac{s\mu(V^+)}{2\k(V^+)V^+}\right | + 
    \frac{\sqrt{\Delta(V^+)}}2,
\end{equation*}
which, in view of  $s<0$, readily implies
\begin{equation}\label{eq:magnitude-u-L} 
    \lambda_u(V^+)>\Lambda_u(V^+)>0,
\end{equation}
and 
\begin{equation}\label{eq:magnitude-s-L} 
    \Lambda_s(V^+)<\lambda_s(V^+)<0.
\end{equation}
In view of the expressions \eqref{eq:tang-lin-L}  and \eqref{eq:tang-tilde-L}, the inequalities \eqref{eq:magnitude-s-L} and \eqref{eq:magnitude-u-L} clearly give a  control of the directions of the stable and unstable eigenspaces of the two systems \eqref{eq:V-Q-reduced} and  \eqref{eq:V-Q} and hence the local behavior of the corresponding stable and unstable manifolds close to the saddle point $(V^+,0)$.  This control will be crucial in the existence proof described in the next Sections.

\begin{remark}
\label{remforshock}
In the case of a forward shock (a Lax 2--shock with $s > \lambda_1(W^-) = - \sqrt{-p'(V^-)} > 0$ and $V^+ > V^-$), it is clear that
\begin{equation}
\label{eq:sign-f-L-forward}
    f(V)<0,\ \mbox{for any}\ V\in(V^-,V^+);\qquad f'(V^+) > 0, \ f'(V^-)<0.
\end{equation}
Henceforth, since $f'(V^-) < 0$ we have $\Delta(V^-) > 0$ and the equilibrium point $(V^-,0)$ is a saddle for system \eqref{eq:V-Q}. In contrast, $(V^+,0)$ is an unstable focus if and only if $\Delta(V^+) < 0$ (or equivalently, if and only if $\eta(V^+) < 2V^+ \sqrt{f'(V^+)}/|s| $); otherwise, $(V^+,0)$ is an unstable node if and only if $\Delta(V^+) \geq 0$ (or equivalently, if and only if $\eta(V^+) \leq 2V^+ \sqrt{f'(V^+)}/|s| $).
\end{remark}

\subsection{Existence of an homoclinic loop for the auxiliary system}\label{subsec:lyap-L}
As the second step in the proof of the existence of the traveling wave, we rigorously demonstrate the existence of a closed  homoclinic orbit for the auxiliary system \eqref{eq:V-Q-reduced}  exiting from  the saddle point $(V^+,0)$.
To this aim, we look for an energy  functional of the form
\begin{equation*}
    H(V,Q):=-\int_{V^+}^V\frac{f(z)}{\k(z)}g_1(z)\,dz+ g_2(V)Q^2,
\end{equation*}
where the functions $g_1$, $g_2$ must be properly chosen such that $H$ is conserved along solutions of \eqref{eq:V-Q-reduced}. 
% that is
% \begin{equation*}
%     \frac{d}{dy}H(R(y),Q(y))=0
% \end{equation*}
% for any $y\in\R$.
% A direct differentiation gives
% $$\frac{\partial H}{\partial R}(R,Q)=\frac{f(R)}{R\k(R)}g_1(R)-g_2'(R)Q^2,$$
% where we used that $f(R^+)=0$, and
% $$\frac{\partial H}{\partial Q}(R,Q)=-2g_2(R)Q.$$
Then, if $(V(y),Q(y))$ solves \eqref{eq:V-Q-reduced}, then we conclude
\begin{align*}
    \frac{d}{dy}H(V(y),Q(y))&=\frac{\partial H}{\partial V}(V,Q)V'(y)+\frac{\partial H}{\partial Q}(V,Q)Q'(y)\\
    &=-\frac{f(V)}{\k(V)}g_1(V)Q+g_2'(V)Q^3\\
        &\quad-2g_2(R)Q\left[\frac{f(V)}{ \k(V)} + \frac{1}{2}  \frac{\k'(V)}{\kappa(V)} Q^2\right].
\end{align*}
Hence, the linear term in $Q$ at the right hand side of the above equality vanishes for $g_1(V)=-2g_2(V)$, while the remaining  one is  zero provided
\begin{equation*}
    g_2'(V)=g_2(V)\frac{\k'(V)}{\kappa(V)},
\end{equation*}
that is $g_2(V)=c\k(V)$,   $c\in\R$. Hence, choosing for definiteness $c=1/2$,  we conclude that
\begin{equation}\label{eq:defH-L}
    H(V,Q):=\int_{V^+}^V f(z)\,dz + \frac{1}{2}\k(V)Q^2,
\end{equation}
satisfies
\begin{equation*}
    \frac{d}{dy}H(V(y),Q(y))=0,
\end{equation*}
for any $y\in\R$, where $(V(y),Q(y))$ solves \eqref{eq:V-Q-reduced}. 
In other words, the orbits of \eqref{eq:V-Q-reduced} are given by
\begin{equation*}
    H(V(y),Q(y))=\bar C
\end{equation*}
for any real constant $\bar C$, and the  homoclinic loop we are looking for is obtained by choosing $\bar C=0$. Indeed, we have already observed (see \eqref{eq:sign-f-L})
that $f$ is negative in $(V^+,V^-)$ and therefore 
\begin{equation*}%\label{eq:F-level0}
     F(V):=\int_{V^+}^Vf(z)\,dz <0,
\end{equation*}
 for any $V\in (V^+,V^-]$.
We claim that there exists a unique $\bar V > V^-$ such that
\begin{equation}\label{eq:FatbarV}
    F(\bar V)=0, %\qquad F(R)\geq0 \qquad \mbox{ if and only if } R\geq\bar R.    
\end{equation}
and, in particular, $F(V)\geq 0$ if and only if $V\geq\bar V$.
Indeed, for $V \geq V^-$ we rewrite
\begin{equation*}
    F(V)=\int_{V^-}^Vf(z)\,dz+\int_{V^+}^{V^-}f(z) \,dz,%=F_1(R)+F(R^-),
\end{equation*}
that is, the sum of a strictly negative constant
\begin{equation}\label{eq:signFV-}
 F(V^-) =  \int_{V^+}^{V^-}f(z)\,dz <0,
\end{equation}
and the integral function  
\begin{equation*}
    V \mapsto \int_{V^-}^V f(z)\,dz.
\end{equation*}
Recalling that $f(z) = p(z) +s^2 z - C$, with $C$ defined in \eqref{def_C-L} such that $f(V^\pm) =0$, and recalling that $p$, and thus $f$, is convex (see \eqref{cond_pressure-L}), then the above integral function is
positive, continuous,   strictly increasing in $(V^-, \infty)$,
and goes to plus infinity for $V\to\infty$. 
%because of \eqref{cond_pressure} and \eqref{eq:f(R)}. 
Thus, there exists a unique $\bar V > V^-$ such that \eqref{eq:FatbarV} holds true, proving the claim.

Summarizing, $H(V,Q)=0$ defines the desired homoclinic loop for \eqref{eq:V-Q-reduced} exiting from the saddle point $(V^+,0)$ and passing  through $(\bar V ,0)$.
The implicit relation defining such orbit can be made  explicit for  $V\in [V^+, \bar V]$ via the definitions
\begin{equation*}
    Q^\pm(V):=\pm\sqrt{-\frac{2F(V)}{\k(V)}}.
\end{equation*}
% and there exist infinitely many solutions of \eqref{eq:R-Q-reduced} satisfying
% \begin{equation}\label{eq:lim-homo}
%     \lim_{y\to-\infty}(R(y),Q(y))=\lim_{y\to+\infty}(R(y),Q(y))=(R^+,0),
% \end{equation}
%with $\displaystyle\min_{\R}R(y)=\bar R$ and $\displaystyle\sup_{\R}R(y)=R^+$.
Finally, the stable and unstable manifolds of the saddle point $(V^+,0)$ for \eqref{eq:V-Q-reduced}  coincide in the region $\left\{(V,Q)\in\R^2 \, : \, V\geq V^+\right\}$ with  the closed curve
\begin{equation}\label{eq:def-Gamma-L}
    \Gamma:=\Gamma^+\cup\Gamma^-, \qquad  
    \Gamma^\pm:=\left\{(V,Q)\in\R^2 \, : \, V\in[V^+, \bar V], \, Q=\pm\sqrt{-\frac{2F(V)}{\k(V)}}\right\}.
\end{equation}
\subsection{Existence of traveling waves}
\label{subsec:existence-L}
In the previous Section we constructed the homoclinic loop $H(V,Q)=0$ for the auxiliary system \eqref{eq:V-Q-reduced} exiting from $(V^+,0)$,  where $H$ is defined in \eqref{eq:defH-L}. Let us now define 
\begin{equation*}
    H^-:=\left\{(V,Q)\in\R^2 \, : \, R\in(V^+, \bar V), \, H(V,Q)<0\right\}.
\end{equation*}
Hence, $H^-$ is the region delimited by the closed curve $\Gamma$, defined in \eqref{eq:def-Gamma-L}, or, equivalently, by the curve $H(V,Q)=0$,  and we readily obtain $(V^-,0)\in H^-$ in view of \eqref{eq:signFV-}. The first result needed to prove the existence of the traveling wave is the invariance of $H^-$ for \eqref{eq:V-Q}. The proof is a direct consequence of the choice of $H$ done in \eqref{eq:defH-L}, that is, a conserved energy for the ``non--viscous'' auxiliary system \eqref{eq:V-Q-reduced}.
\begin{lemma}\label{lem:neg-invariant-L}
The region $H^-$ is negatively invariant for the system \eqref{eq:V-Q}, i.e. if $(V(y),Q(y))$ is a solution of \eqref{eq:V-Q} satisfying $(R(y_0),Q(y_0))\in H^-$ for some $y_0\in\R$, then
$(R(y),Q(y))\in H^-$ for any $y\leq y_0$.  
\end{lemma}
\begin{proof} As we have already emphasized above, the functional $H$ is constructed so that it is conserved along trajectories of the dynamical system \eqref{eq:V-Q-reduced}. Hence, if we consider now $(V(y),Q(y))$ to be a solution of \eqref{eq:V-Q}, we readily obtain 
\begin{align*}
    \frac{d}{dy}H(V(y),Q(y))&=\frac{\partial H}{\partial V}(V,Q)V'(y)+\frac{\partial H}{\partial Q}(V,Q)Q'(y)\\
    &=f(V)Q+\frac12\ k'(V)Q^3\\
        &\quad +\k(V)Q\left[ - \frac{f(V)}{\k(V)} - \frac{s \mu(V)}{V \k(V)} Q - \frac{1}{2}  \frac{\k'(V)}{\k(V)}Q^2\right]\\
    &=-\frac{s\mu(V)}{V}Q^2\geq0,
\end{align*}
for any $y\in\R$, being $s<0$ for a 1--shock (backward shock).
As a consequence, if $(R(y_0),Q(y_0))\in H^-$ for some $y_0\in\R$, then
\begin{equation*}
    H(R(y),Q(y))\leq H(R(y_0),Q(y_0))<0
\end{equation*}
for any $y\leq y_0$ and the proof is complete.
\end{proof}

The second preliminary result concerns the comparison between the directions of the  stable and unstable manifolds corresponding to the saddle point $(V^+,0)$ for \eqref{eq:V-Q} with those of this point for the reduced system \eqref{eq:V-Q-reduced}, the latter being given by the  homoclinic loop $H(V,Q)=0$ as explained above. 
\begin{lemma}\label{lem:comp-manifolds-L}
The stable (resp.\ unstable) manifold at $(V^+,0)$ for \eqref{eq:V-Q}  is pointing  inside (resp.\ outside) $H^-$.
\end{lemma}
\begin{proof}
The geometric properties we need to obtain the stated results are a direct consequence  of the inequalities \eqref{eq:magnitude-u-L} and \eqref{eq:magnitude-s-L},  which in turn compare the tangent vectors $\nu_u(V^+)$ with  $\tau_u(V^+)$ and  $\nu_s(V^+)$ with $\tau_s(V^+)$, defined in \eqref{eq:tang-tilde-L} and \eqref{eq:tang-lin-L}. The relative positions of the two manifolds are  depicted in Figure \ref{fig:comp-manifolds-L}.
\begin{figure}
\centering
\includegraphics{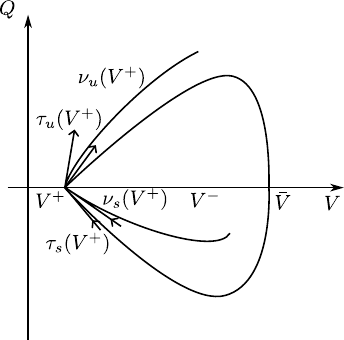}
\caption{Stable and unstable manifolds for \eqref{eq:V-Q} and \eqref{eq:V-Q-reduced} at $(V^+,0)$.}
\label{fig:comp-manifolds-L}
\end{figure}
\end{proof}

We are now able to prove the main result of this Section.
\begin{theorem}
\label{theo:ex-L}
Assume that the end states $(V^\pm,U^\pm)$ and the speed $s$ define a compressive $1$--shock (backward shock) for the $p$--system \eqref{eq:Euler-Lag}. Suppose that the viscosity coefficient $\mu$ and $\k$ are smooth, positive functions of $v>0$ and the pressure verifies \eqref{cond_pressure-L}. Then there exists a unique (up to translations) traveling profile solving \eqref{eq:profiles-end-Lag} and \eqref{eq:profiles-L}. Moreover, this traveling wave 
%is monotone decreasing ($V' < 0$, $U' < 0$) when 
%\[
%\eta(V^-) \geq \frac{2V^- \sqrt{f'(V^-)}}{|s| },
%\]    
%and it 
is oscillating whenever
\begin{equation}
\label{osccondb}
0 < \eta(V^-) < \frac{2 V^-\sqrt{f'(V^-)}}{|s| }.
\end{equation}
\end{theorem}
\begin{proof}
First of all, we claim that that the stable manifold of $(V^+,0)$ for \eqref{eq:V-Q} is entirely contained inside $H^-$. Indeed,  we know that the point $(V(y),Q(y))$ on the stable manifold belongs to that set for any $y\geq y_0$, for $y_0$ sufficiently big, thanks to  Lemma \ref{lem:comp-manifolds-L}. Then, we can follow this trajectory backward for $y\geq y_0$ and, using this time Lemma \ref{lem:neg-invariant-L}, we conclude. 

As a consequence,  the desired result is clearly obtained provided that, along this trajectory, we have
\begin{equation*}
    \lim_{y\to-\infty}(V(y), Q(y))=(V^-,0),
\end{equation*}
which gives an heteroclinic connection for \eqref{eq:V-Q} joining the equilibria  $(V^-,0)$ and $(V^+,0)$.
This last step will be  obtained via an appropriate  Lyapunov functional and using the classical LaSalle's invariance principle (cf. \cite{Teschl12}). For this, let us consider the following functional
\begin{equation*}
L(V,Q):= H(V,Q) - F(V^-),
\end{equation*}
and let $(V(y), Q(y))$ be an orbit of \eqref{eq:V-Q}. Then, $L(R^-,0) = 0$ and, inverting the direction along the orbit,  we have 
\begin{equation*}
\frac{d}{dy} L(V(-y), Q(-y))=  - \frac{d}{d\tilde y}H(V(\tilde y),Q(\tilde y)) \leq 0,
\end{equation*}
namely,   $L$  is a Lyapunov functional for that system in the ``reverse dynamics'' $-y$. Since the set of points where the above derivative vanish is given by $\{Q=0\}$, and since the trajectory is confined in $H^-$, the LaSalle's invariance principle implies that any solution $(V(-y),Q(-y))$ in $\overline{H^-}$ shall converge as $-y\to \infty$ (equivalently, as $y\to -\infty$)  to the largest invariant subset of $\overline{H^-}\cap\{Q=0\}$, which is the set of two steady--states $(V^+,0)$ and $(V^-,0)$. Finally,  as proven before, $(V^+,0)$ can not be reached for $-y\to \infty$ (equivalently, $y\to -\infty$)  by orbits inside $\overline{H^-}$, because this should occur along the unstable manifold of this saddle point, and the latter  is excluded by Lemma \ref{lem:comp-manifolds-L}.  Hence, we conclude that the trajectory under examination   exiting along the stable manifold  of $(V^+,0)$  converges as $y\to -\infty$ to the stable point $(V^-,0)$, proving the existence of the  heteroclinic connection and, in turn, of the desired traveling profile. Finally, it is clear from the calculations of Section \ref{subseclinfull} that condition \eqref{osccondb} holds if and only if the equilibrium point $(V^-,0)$ is an unstable focus and the profile is oscillating. Otherwise, $(V^-,0)$ is an unstable node and the profiles are monotone. The Theorem is now proved.
%\hole{(E) TO DO: Finish with a comment on monotonicity and oscillating threshold.}
\end{proof}

A symmetric result in the case when the triplet $(V^\pm,U^\pm, s)$ defines a compressive $2$--forward shock is also true.

\begin{theorem}\label{theo:ex-2-L}
    Assume that the end states $(V^\pm,U^\pm)$ and the speed $s$ define a compressive $2$--shock (forward shock) for the $p$--system \eqref{eq:Euler-Lag}. Suppose that the viscosity and capillarity coefficients, $\mu$ and $\k$, are smooth, positive functions of $v>0$ and that the pressure verifies \eqref{cond_pressure-L}. Then, there exists a unique (up to translations) traveling profile solving \eqref{eq:profiles-end-Lag} and \eqref{eq:profiles-L}. Moreover, this traveling wave 
    %is monotone increasing in the $V$-variable $(V' > 0)$ and monotone decreasing in the $U$-variable ($U' < 0$), provided that
%\[
%\eta(V^+) \geq \frac{2V^+ \sqrt{f'(V^+)}}{|s|},
%\] 
%and it 
is oscillating whenever
\begin{equation}
\label{osccondf}
0 < \eta(V^+) < \frac{2 V^+\sqrt{f'(V^+)}}{|s|}.
\end{equation}       
\end{theorem}
The proof of this result is analogous to the previous one; we sketch here only the main differences. In Figure \ref{fig:comp-manifolds-L}, the points $(V^-,0)$ and  $(V^+,0)$ are exchanged (remember in this case $V^- < V^+$), and the region $H^-$ is positively invariant for \eqref{eq:V-Q}, being $s>0$. Then, Lemma \ref{lem:comp-manifolds-L} is still valid, clearly referring to the saddle points  $(V^-,0)$ for the complete and reduced systems. Finally, the existence of the  heteroclinic connection is again a direct consequence of the existence of a Lyapunov functional and  the LaSalle's invariance principle, this time for the ``forward dynamics'' of \eqref{eq:V-Q}. Like in the previous case, condition \eqref{osccondf} holds if and only if $(V^+,0)$ is an unstable focus and the profiles are oscillating (see Remark \ref{remforshock}).
\begin{remark}
\label{rem:osc-L}
    It is worth underlying once again that the above existence results include the case of non monotone profiles, which is clearly the case when the equilibrium point inside the region $H^-$, namely, $(V^-,0)$ for compressive $1$--shocks (backward shocks),  $(V^+,0)$ for compressive $2$--shocks (forward shocks),  is a stable focus; see  
    \eqref{osccondb} and \eqref{osccondf}.
    These conditions can not be verified for sufficiently small shocks, since in that case $|f'(V^\pm)|\ll 1$ and the point inside the region is a stable node. The detailed calculations, and also  the proof of the monotonicity of the profile for small shocks, together with other relevant properties, are the content of Section \ref{sec:properties-small-profiles-L}.
\end{remark}

\subsection{Numerical calculation of the profiles}
\label{secnumerics}
For illustrative purposes, we present the numerical calculation of a viscous-dispersive shock profile for system \eqref{eq:NSK-Lag}. For concreteness, we consider a viscous-capillar compressible fluid endowed with the following nonlinear pressure, viscosity and capillarity coefficients,
\begin{equation}
\label{numexample}
p(v) = \frac{1}{v^{5/3}}, \qquad \mu(v) = \frac{1.2}{v^2}, \qquad \kappa(v) = \frac{10}{v^7},
\end{equation}
defined for all positive values of the specific volume $v \geq \delta > 0$, with some uniform $\delta > 0$. These choices correspond to an adiabatic $\gamma$-gas law with $\gamma = 5/3 > 1$ and encompass typical functions of algebraic form for the viscosity and capillarity coefficients which can be found in the literature (see, for example, \cite{CCD15}). The end states are taken as
\[
V^+=1, \quad V^- = 1.5,
\]
so that the shock speed can be computed through \eqref{eq:RH-rewr-L} with $s <  \lambda_1(W^+) = - \sqrt{p(V^+)} < 0$, that is, a compressive $1$-shock. Hence, $s \approx -0.9912 <0$. With these parameter values we may compute 
\[
\Delta(V^-) = \left[\frac{s\mu(V^-)}{\k(V^-)V^-}\right]^2-\frac{4f'(V^-)}{\k(V^-)} \approx - 1.3062 < 0,
\]
implying that the equilibrium point $(V^-,0)$ is an unstable focus and, consequently, the viscous-dispersive shock profile is oscillatory. The initial condition at $x = 0$ is set as 
\[
(V(0),Q(0)) = (V^+,0) + \epsilon (1, \lambda_s(V^+)),
\]
with $\epsilon = 0.000001$, so that it belongs to the stable manifold at the end state $V^+$. The heteroclinic trajectory of the planar system \eqref{eq:V-Q} connecting $(V^-,Q^-) = (1.5,0)$ with $(V^+,Q^+) = (1,0)$ can be computed with a standard explicit Euler time-integrator. 
Figure \ref{fig:phasespace} shows the heteroclinic trajectory in the phase $(V,Q)$-space in red color.
\begin{figure}[t]
\centering
\includegraphics[scale=.5, clip=true]{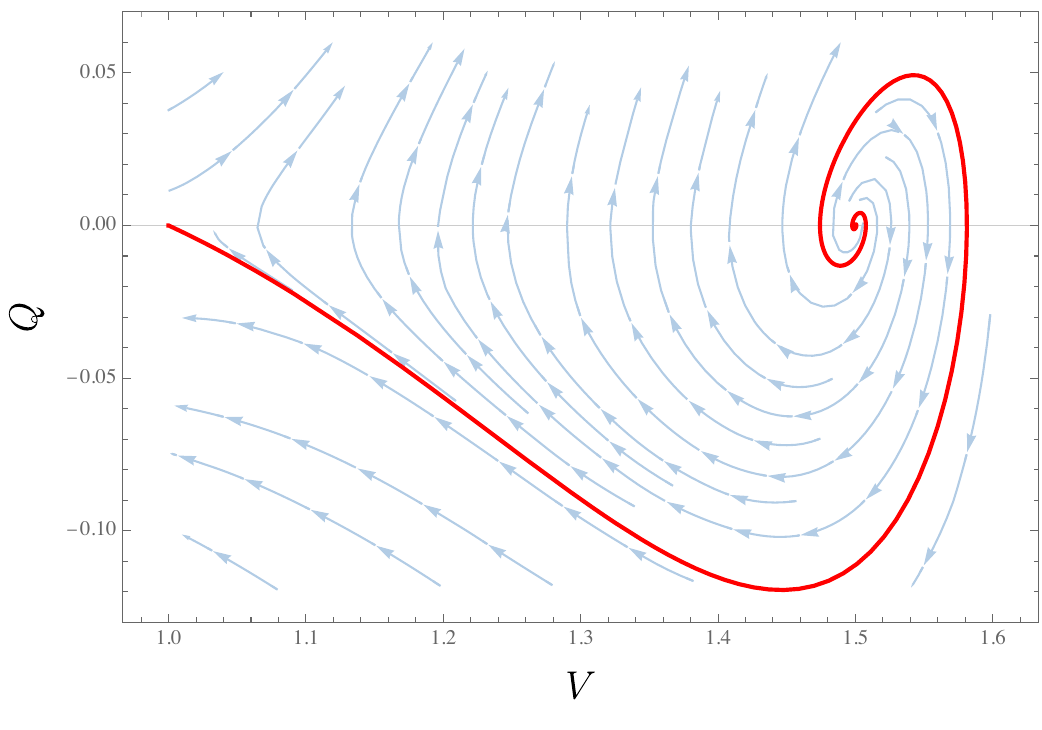}
\caption{\small{Plot of the flow of the planar system \eqref{eq:V-Q} in the phase space. The viscous-dispersive shock profile corresponds to the heteroclinc trajectory in red color. The flow of the system in the $(V,Q)$-space is represented in light blue color (color online).}}
\label{fig:phasespace}
\end{figure}
The flow of system \eqref{eq:V-Q} is depicted in light blue color. The curve leaves the equilibrium $(V^-,Q^-) = (1.5, 0)$ along the unstable manifold at that point and reaches the second equilibrium, namely $(V^+,Q^+) = (1,0)$, along its stable manifold. Notice that this is an oscillatory, non-monotone, viscous-dispersive shock profile. The profile function for the specific volume, $x \mapsto V(x)$, is depicted in Figure \ref{fig:Vprofile} in red color. The non-monotonicty and oscillatory properties of the profile are manifest in this case.

\begin{figure}[t]
\centering
\includegraphics[scale=.5, clip=true]{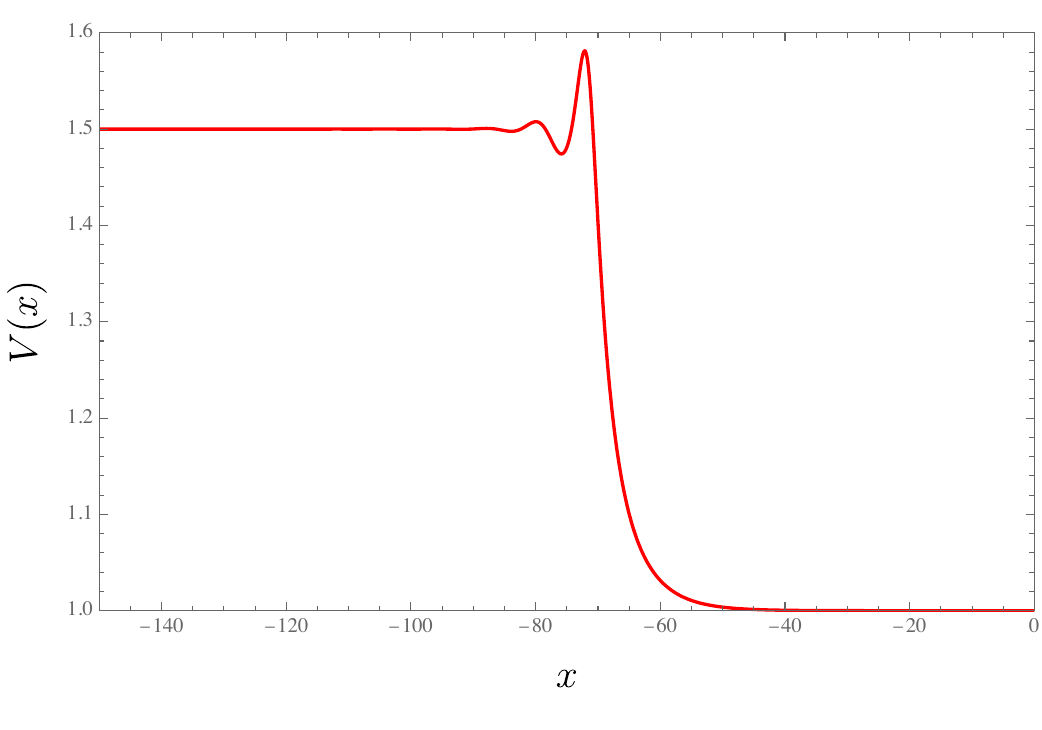}
\caption{\small{Profile $x \mapsto V(x)$ (in red color) for the specific volume of the viscous-dispersive shock profile (color online).}}
\label{fig:Vprofile}
\end{figure}

\section{Further properties of viscous-dispersive shock profiles}
\label{sec:properties-small-profiles-L}

In this Section we analyze the properties of the profiles constructed in the previous one in detail and in the particular case of small shocks, that is for $|V^+ - V^-|$ (and hence $|U^+ - U^-|$ as well; see \eqref{RH1-L}),  sufficiently small. To fix ideas, let us consider the case of a Lax 1--shock (backward shock), that is 
\begin{equation*}
    \lambda_1(W^+) < s < \lambda_1(W^-)<0,
\end{equation*}
which implies $ V^- > V^+$ (see \eqref{cond_1_shock-L}-\eqref{eq:V->V+1shock-L}); the (symmetric) case of a Lax 2--shock is analogous and will be stated at the end of the Section (see Theorem \ref{theo:monotone-2-L} below). In what follows, we outline the main properties of the profiles for $0<\varepsilon \ll 1$, $\varepsilon := V^- - V^+$.

The first result we prove states that, for small Lax 1--shocks, the profile $V$ is (strictly) monotone decreasing. In turn, from \eqref{eq:profiles-L} we also conclude
\begin{equation*}
    U'  = -sV'<0.
\end{equation*}
The above property for $V$ is clearly in agreement with the linear analysis carried out in Section \ref{subsec:DynSystLinear-L} for the dynamical system \eqref{eq:V-Q}. Indeed, since
\begin{equation*}
    f'(V) = p'(V)+s^2,
\end{equation*}
from \eqref{eq:RH-rewr-L}, it follows that
\begin{equation*}
    f'(V^-) = p'(V^-)+\frac{p(V^-)-p(V^+)}{V^+-V^-}.
\end{equation*}
Hence, if $V^--V^+=\e$ with $0<\e\ll1$, then we conclude $0< f'(V^-) = O(\varepsilon)$
and the discriminant $\Delta(V^-)$ given by  \eqref{eq:discriminant-L} is positive.
As a consequence, $(V^-, 0)$ is an unstable node, which is a necessary condition for $V$ to be monotone decreasing. 
\begin{theorem}\label{theo:monotone-L}
    Assume  the end states $(V^\pm,U^\pm)$ and the speed $s$ define a  $1$--shock (backward shock) for the $p$--system \eqref{eq:Euler-Lag} and assume its amplitude is sufficiently small: $0< \varepsilon = V^- - V^+ \ll 1$.  Then, the profile $V$ constructed in Theorem \ref{theo:ex-L} is monotone decreasing and it satisfies the following properties for any $y\in\mathbb{R}$:
    \begin{align}\label{eq:estprof-L}
    &-C\e^2<V'(y)<0,\\
    & |V''(y)|\leq C\e |V'(y)|,\label{eq:estprof-2-L}\\
    & |V'''(y)|\leq C\e^2 |V'(y)|,\label{eq:estprof-3-L}
\end{align}
where $C$ is a positive constant independent from $\varepsilon$.
\end{theorem}
\begin{proof}   
By using the definition of $f$ in \eqref{eq:def-f-Lag} and \eqref{eq:RH-rewr-L}-\eqref{def_C-L}, we obtain
\[
    f(V)=p(V)-p(V^-)+s^2(V-V^-)=p(V)-p(V^-)+\frac{p(V^-)-p(V^+)}{V^+-V^-}(V-V^-).
\]
%\begin{align*}
%    f(V)&=p(V)-p(V^-)+s^2(V-V^-)\\
%    &=p(V)-p(V^-)+\frac{p(V^-)-p(V^+)}{V^+-V^-}(V-V^-).
%\end{align*} 
Hence, if $V^+=V^--\varepsilon$, one has
\begin{equation*}
    f(V)=p(V)-p(V^-)+\frac{p(V^--\e)-p(V^-)}{\e}(V-V^-).
\end{equation*}
Set
\begin{equation}\label{eq:def-z-tildez-L}
    V(y)= : \e z(\e y)+\frac12(V^-+V^+),\qquad 
    \tilde z(x): =z(x)-\frac12,
\end{equation}
so that $V(y)=V^-+\e\tilde z(\e y)$ and $V'(y)=\e^2 \dot{\tilde z}(\e y)$, where ``$\dot{\;\;}$" denotes the derivative with respect to the variable $x:=\e y$.
Notice that $V(y) \to V^\pm$ for $y\to\pm\infty$ becomes
\begin{equation*}
    \lim_{x\to\pm\infty} z(x)=\mp\frac12.
\end{equation*}
In the new independent variable $x=\e y$ \eqref{eq:V-Q} rewrites as the following first order system
\begin{equation}\label{eq:z-w-L}
    \begin{cases}
        \dot{\tilde z} = w,\\
        \e \dot w = \tilde{F}(\tilde z,w,\e),
    \end{cases}     
\end{equation}
where 
\begin{align*}
    \tilde{F}(\tilde z,w,\e):=&-\frac{ \tilde f(\e)}{\e^2\k(V^-+\e\tilde z)} + \frac{\mu(V^-+\e\tilde z)}{(V^-+\e\tilde z)\k(V^-+\e\tilde z)}\sqrt{\frac{p(V^--\e)-p(V^-)}{\e}} w \\
    &- \frac{\e^2}{2} \left (\frac{\k'(V^-+\e\tilde z)}{\k(V^-+\e\tilde z)}\right ) w^2,
\end{align*}
with
\begin{equation*}
    \tilde f(\e):=f(V^-+\e\tilde z)=p(V^-+\e\tilde z)-p(V^-)+[p(V^--\e)-p(V^-)]\tilde z.
\end{equation*} 
Notice that $\tilde f(0)=0$ and 
\begin{equation*}
   \frac{d}{d\e} \tilde f(\e)\bigg |_{\e=0}=p'(V^-+\e\tilde z)\tilde z-p'(V^--\e)\tilde z\bigg |_{\e=0} = 0.
\end{equation*}
which implies $\tilde f'(0)=0$.
Moreover, a second differentiation gives
\begin{align*}
     \frac{d^2}{d\e^2}{f}''(\e)\bigg |_{\e=0}=p''(V^-)\tilde{z}^2+p''(V^-)\tilde z.
\end{align*}
As a consequence, the Taylor expansion of $\tilde f$ at zero is given by
\begin{equation}\label{eq:tildef-eps-L}
    \tilde f(\e)=\frac12p''(V^-)\tilde z(1+\tilde z)\e^2+\mathcal{O}(\e^3).
\end{equation}
Upon substitution of \eqref{eq:tildef-eps-L} in the definition of $\tilde F$, we deduce
\begin{equation*}
    \tilde{F}(\tilde z,w,0):=-\frac{\tilde z(1+\tilde z)p''(V^-)}{2\k(V^-)}+\frac{\mu(V^-)\sqrt{-p'(V^-)}}{V^-\k(V^-)} w.
\end{equation*}
Therefore, by following the classical geometric singular perturbation theory (see, for instance, \cite{J95}), we conclude that the manifold $\mathcal{M}_0 := \left \{ \tilde{F}(\tilde z,w,0)=0\right \} $ is normally hyperbolic, and, for $\varepsilon\ll 1$, there exists 
 the slow manifold of \eqref{eq:z-w-L}, lying close to  $\mathcal{M}_0$, and it is given by 
\begin{equation*}
    \mathcal{M}_\e:=\left\{(\tilde z,w)\, : \, w=h_\e(\tilde z)\right\},     
\end{equation*}
where
\begin{equation*}
    h_\e(\tilde z):=h_0(\tilde z)+\e h_1(\tilde z,\e), \qquad \quad  h_0(\tilde z):=\frac{\tilde z(1+\tilde z)p''(V^-)V^-}{2\mu(V^-)\sqrt{-p'(V^-)}},
\end{equation*}
and $h_1$ is a smooth function.
Hence, the equation on the slow manifold becomes
\begin{equation}\label{eq:ODE-slow-L}
    \dot{\tilde z}=C(V^-)\tilde z(1+\tilde z)+\e h_1(\tilde z,\e),
\end{equation}
where
\begin{equation*}
    C(V^-):=\frac{p''(V^-)V^-}{2\mu(V^-)\sqrt{-p'(V^-)}}>0.
\end{equation*}
The leading order equation in \eqref{eq:ODE-slow-L} is given by 
\begin{equation*}
    \dot{\tilde z}=C(V^-)\tilde z(1+\tilde z),
\end{equation*}
and since $C(V^-)>0$, if $\e$ is sufficiently small, then $\tilde z$ is monotone decreasing and satisfies
\begin{equation*}
    \lim_{x\to-\infty} \tilde z(x)=0, \qquad \qquad  \lim_{x\to\infty} \tilde z(x)=-1.
\end{equation*}
As a consequence, from \eqref{eq:def-z-tildez-L} it follows that the profile $V$ is monotone decreasing and, since $V'(y)=\e^2\dot{\tilde z}(\e y)$, we end up with   \eqref{eq:estprof-L}. 
Moreover, since
\begin{equation*}
    |\ddot{\tilde z}|=\left|h_0'(\tilde z)+\e\partial_{\tilde z}h_1(\tilde z,\e)\right||\dot{\tilde z}|\leq C|\dot{\tilde z}|,
\end{equation*}
for some $C>0$, one has \eqref{eq:estprof-2-L}.
Similarly, since
\begin{equation*}
    |\dddot{\tilde z}|=\left|h_0''(\tilde z)+\e\partial_{\tilde z\tilde z}h_1(\tilde z,\e)\right|\dot{\tilde z}^2+\left|h_0'(\tilde z)+\e\partial_{\tilde z}h_1(\tilde z,\e)\right||\ddot{\tilde z}|\leq C|\dot{\tilde z}|,
\end{equation*}
and $V'''(y)=\e^4\dddot{\tilde z}(\e y)$, we obtain \eqref{eq:estprof-3-L} and the proof is complete.
\end{proof}

The symmetric case of a small $2$--shock (forward shock) is analogous; the details of the proof are left to the reader.
\begin{theorem}\label{theo:monotone-2-L}
    Assume  the end states $(V^\pm,U^\pm)$ and the speed $s$ define a  $2$--shock (forward shock) for the $p$--system \eqref{eq:Euler-Lag} and assume its amplitude is sufficiently small: $0< \varepsilon = V^+ - V^+ \ll 1$.  Then, the profile $V$ constructed in Theorem \ref{theo:ex-2-L} is monotone increasing and it satisfies the following properties for any $y\in\mathbb{R}$:
    \begin{align*}%\label{eq:estprof-second}
    &0<V'(y)<C\e^2,\\
    & |V''(y)|\leq C\e |V'(y)|,\\%\label{eq:estprof-2-second}
     & |V'''(y)|\leq C\e^2 |V'(y)|,%\label{eq:estprof-3}
\end{align*}
where $C$ is a positive constant independent from $\varepsilon$.
\end{theorem}

\section{Linearization and stability of the essential spectrum}
\label{sec:essential-L}

In this Section we examine the linearization of the system \eqref{eq:NSK-Lag} around the viscous-dispersive profile $(U,V)$ defined in Section \ref{sec:prel-L} and whose existence has been proved in Section \ref{sec:ex-L}. Moreover, we prove that, under very general conditions, the essential spectrum of the linearized operator around the profile is stable.

\subsection{The linearized operator around the profile}

Let us first recall some standard definitions. For any closed, densely defined operator, $\cL \in \ccC(X,Y)$, with $X$, $Y$ Banach spaces, its \textit{resolvent set}, $\rho(\cL)$, is defined as the set of all complex numbers $\lambda \in \C$ such that $\cL - \lambda$ is  injective, $\cR(\cL - \lambda) = Y$ and $(\cL - \lambda)^{-1}$ is a bounded operator. Its \textit{spectrum} is defined  as $\sigma(\cL) := \C \backslash \rho(\cL)$. For nonlinear wave stability purposes (cf. \cite{KaPro13}), the spectrum is often partitioned into essential spectrum, $\ess(\cL)$, and point spectrum,  $\ptsp(\cL)$, being the former the set of  complex numbers $\lambda$ for which $\cL - \lambda$ is either not  Fredholm, or is Fredholm with index different from zero, whereas $\ptsp(\cL)$ is defined as the set of complex numbers for  which  $\cL - \lambda$ is Fredholm with index zero and has a non-trivial kernel. This definition is due to Weyl \cite{We10}, making $\ess$ a large set but easy to compute, whereas $\ptsp$ comprises discrete eigenvalues with finite algebraic multiplicities (see Remark 2.2.4 in \cite{KaPro13}). It is to be observed that, since the operator is closed, then $\sigma(\cL) = \ptsp(\cL) \cup \ess(\cL)$ (for a quick proof of this fact see \cite{FNP24}, Section 3.1; the reader is referred to Kato \cite{Kat80} and Kapitula and Promislow \cite{KaPro13} for further information).

Let us consider solutions to the NSK system \eqref{eq:NSK-Lag} of the form 
\[
(v,u)(x,t) = (V(y) + \tiv(y,t), U(y) + \tiu(y,t)), 
\]
where $y = x-st$ denotes the translation variable and $(U,V)(y)$ is the viscous-dispersive shock profile under consideration. The new variables $(\tiv,\tiu)$ represent perturbations of the latter. Upon substitution into \eqref{eq:NSK-Lag} and linearizing around the shock profile, we obtain a linear system of equations for the perturbations of the form
\[
\partial_t \begin{pmatrix} \tiv \\ \tiu \end{pmatrix} = \tcL \begin{pmatrix} \tiv \\ \tiu \end{pmatrix},
\]
where $\tcL$ denotes the linearized operator around the profile and it is given by
\begin{equation}
\label{deflinop}
    \tcL \begin{pmatrix}\tiv\\ \tiu\end{pmatrix} := \begin{pmatrix}s\tiv_y+\tiu_y\\
    s\tiu_y-\left(p'(V)\tiv\right)_y +\tcL_v\begin{pmatrix}\tiv\\ \tiu\end{pmatrix}+\tcL_c \tiv \end{pmatrix},
\end{equation}
with
\[
\tcL_v\begin{pmatrix}\tiv\\ \tiu\end{pmatrix} :=\left[\frac{\mu(V)}{V}\tiu_y+\frac{\mu'(V)V-\mu(V)}{V^2}U'\tiv\right]_y,
\]
\[
 \tcL_c \tiv :=-\left[\k(V)\tiv_{yy}+\k'(V)
    V''\tiv+ \k'(V)V'\tiv_y+\frac12 \k''(V)(V')^2\tiv\right]_y.
\]
%\[
%\begin{aligned}
%\tcL_v\begin{pmatrix}\tiv\\ \tiu\end{pmatrix} &:=\left[\frac{\mu(V)}{V}\tiu_y+\frac{\mu'(V)V-\mu(V)}{V^2}U'\tiv\right]_y,\\
% \tcL_c \tiv &:=-\left[\k(V)\tiv_{yy}+\k'(V)
%    V''\tiv+ \k'(V)V'\tiv_y+\frac12 \k''(V)(V')^2\tiv\right]_y.
%\end{aligned}
%\]
%\begin{equation*}
%    \mathcal{L}_v\begin{pmatrix}\tiv\\ \tiu\end{pmatrix}:=\left[\frac{\mu(V)}{V}\tiu_y+\frac{\mu'(V)V-\mu(V)}{V^2}U'\tiv\right]_y,
%\end{equation*}
%and 
%\begin{equation*}
%    \mathcal{L}_c(\tiv):=-\left[\k(V)\tiv_{yy}+\k'(V)
%    V''\tiv+ \k'(V)V'\tiv_y+\frac12 \k''(V)(V')^2\tiv\right]_y.
%\end{equation*}

Motivated by the notion of spatially localized, finite energy perturbations in the translation coordinate frame in which the wave is stationary, let us consider $X = L^2 \times L^2$ as a base space so that the operator $\tcL : L^2 \times L^2 \to L^2 \times L^2$ is a closed, densely defined operator with domain $\cD(\tcL) = H^3 \times H^2$. Likewise the auxiliary operators $\tcL_v : \cD(\tcL_v) = H^1 \times H^2 \subset L^2 \times L^2 \to L^2$ and $\tcL_c : \cD(\tcL_c) = H^3 \to L^2$ are closed and densely defined (see, e.g., \cite{KaPro13}). The associated eigenvalue problem for the linearized operator $\tcL$ is therefore given by
\begin{equation}
\label{eq:eigenvalue-L-Lag}
    \tcL\begin{pmatrix}\tiv\\ \tiu\end{pmatrix}=\lambda\begin{pmatrix}\tiv\\ \tiu\end{pmatrix},
\end{equation}
for some $\lambda \in \C$, $(\tiv, \tiu) \in H^3 \times H^2$ and the standard partition of the spectrum applies. We now recall the following standard definition.
\begin{definition}
\label{defspecstab}
We say that the viscous-dispersive shock profile $(V,U)$ is \emph{spectrally stable} if 
 \[ 
 \sigma(\tcL) \subset \{ \lambda \in \C \, : \, \Re{\lambda} \leq 0\}.
  \]
\end{definition}

\begin{remark}
\label{transevalue}
As it is customary in the stability analysis of traveling waves, the translation invariance of the profile induces an eigenvalue located at the origin. In other words, $\lambda = 0 \in \ptsp(\tcL)$ with eigenfunction given by $(V',U') \in H^3 \times H^2$ inasmuch as $\tcL (V',U') = 0$, as the reader may directly verify. That $(V',U') \in \cD(\tcL)$ follows from the exponential decay of the profiles (which is, in turn, a consequence of the hyperbolicity of the equilibrium points $(V^\pm, U^\pm)$). We omit the details.
\end{remark}

\subsection{Stability of the essential spectrum}
First, let us examine the location of the essential spectrum. To that end, we introduce the asymptotic operators $\tcL_\pm$, obtained by taking the limits when $y\to\pm\infty$ in \eqref{deflinop} and recalling that
\begin{equation*}
    \lim_{y\to\pm\infty}V(y)=V^\pm, \qquad \qquad \lim_{y\to\pm\infty}U(y)=U^\pm,
\end{equation*}
and
\begin{equation*}
    \lim_{y\to\pm\infty}V'(y)=\lim_{y\to\pm\infty}V''(y)=
    \lim_{y\to\pm\infty}U'(y)=0.
\end{equation*}
In particular, we derive the eigenvalue problem associated to $\tcL_\pm$,
\begin{equation*}
    \tcL_\pm\begin{pmatrix}\tiv\\ \tiu\end{pmatrix}=\lambda\begin{pmatrix}\tiv\\ \tiu\end{pmatrix},
\end{equation*}
where
\begin{equation*}
     \tcL_\pm\begin{pmatrix}\tiv\\ \tiu\end{pmatrix}:=
     \begin{pmatrix}
         s\tiv'+\tiu'\\
         \displaystyle{s\tiu'-p'(V^\pm)\tiv'+\frac{\mu(V^\pm)}{V^\pm}\tiu''-\k(V^\pm)\tiv'''}
     \end{pmatrix},
\end{equation*}
and $' = d/dy$ for shortness. The latter eigenvalue problem can be written as the first order system $W' = M_\pm(\lambda) W$, with 
\begin{equation*}
    W=\begin{pmatrix}
        \tiv\\ \tiu\\ \tiv'\\ \tiv''
    \end{pmatrix}
    \qquad \mbox{ and } \qquad
    M_\pm(\lambda)=\begin{pmatrix}
        0 & 0 & 1 & 0\\
        \lambda & 0 & -s & 0\\
        0 & 0 & 0 & 1\\
        \displaystyle\frac{\lambda s}{\k(V^\pm)} & -\displaystyle\frac{\lambda}{\k(V^\pm)} & \alpha_\pm(\lambda) & -\displaystyle\frac{s\mu(V^\pm)}{V^\pm\k(V^\pm)}
    \end{pmatrix},
\end{equation*}
where we have defined
\begin{equation*}
    \alpha_\pm(\lambda):=\frac{\lambda\mu(V^\pm)}{V^\pm\k(V^\pm)}-\frac{s^2+p'(V^\pm)}{\k(V^\pm)}.
\end{equation*}
Let us compute the characteristic equation $\det(M_\pm(\lambda)-\theta I)=0$, that is
\begin{equation*}
    -\det\begin{pmatrix}
       -\theta & 0 & 1\\
       \lambda & -\theta & -s\\
       \displaystyle\frac{\lambda s}{\k(V^\pm)} & -\displaystyle\frac{\lambda}{\k(V^\pm)} & \alpha_\pm(\lambda)
    \end{pmatrix}-\left(\frac{s\mu(V^\pm)}{V^\pm\k(V^\pm)}+\theta\right)\det\begin{pmatrix}
        -\theta & 0 & 1\\
        \lambda & -\theta & -s\\
        0 & 0 & -\theta
    \end{pmatrix}=0.
\end{equation*}
A direct computation yields
\begin{equation*}
    \det\begin{pmatrix}
       -\theta & 0 & 1\\
       \lambda & -\theta & -s\\
       \displaystyle\frac{\lambda s}{\k(V^\pm)} & -\displaystyle\frac{\lambda}{\k(V^\pm)} & \alpha_\pm(\lambda)
    \end{pmatrix}=\theta\left(\theta\alpha_\pm(\lambda)+\displaystyle\frac{\lambda s}{\k(V^\pm)}\right)-\frac{\lambda^2}{\k(V^\pm)}+\frac{\lambda s\theta}{\k(V^\pm)},
\end{equation*}
and
\begin{equation*}
    \det\begin{pmatrix}
        -\theta & 0 & 1\\
        \lambda & -\theta & -s\\
        0 & 0 & -\theta
    \end{pmatrix}=-\theta^3.
\end{equation*}
As a consequence,
\begin{equation}
\label{detMtheta}
    \det(M_\pm(\lambda)-\theta I)=\theta^4+\frac{s\mu(V^\pm)}{V^\pm\k(V^\pm)}\theta^3-\alpha_\pm(\lambda)\theta^2-\frac{2\lambda s}{\k(V^\pm)}\theta+\frac{\lambda^2}{\k(V^\pm)}.
\end{equation}
Setting $\theta=i\xi$, with $\xi\in\R$ and multiplying by $\k(V^\pm)$ the characteristic equation $\det(M_\pm(\lambda)- i\xi I)=0$, we deduce the dispersion relation
\begin{equation}
\label{eq:dispersion-Lag}
d_\pm(i\xi, \lambda) = 0,
\end{equation}
with
\[
d_\pm(i\xi, \lambda)  :=\lambda^2+\left(\frac{\xi^2\mu(V^\pm)}{V^\pm}-2\xi si\right)\lambda-\left[s^2+p'(V^\pm)\right]\xi^2-\frac{s\mu(V^\pm)}{V^\pm}i\xi^3+\k(V^\pm)\xi^4,
\]
and where we have substituted the definition of $\alpha_\pm(\lambda)$.
The discriminant of \eqref{eq:dispersion-Lag} is given by
\begin{align*}
    \widetilde{\Delta}&:=\widetilde{\Delta}(V^\pm,\xi)=\frac{\xi^4\mu(V^\pm)^2}{(V^\pm)^2}+4p'(V^\pm)\xi^2-4\k(V^\pm)\xi^4.
\end{align*}
Hence, the solutions $\lambda_{1,2}^\pm(\xi):=\lambda_{1,2}(V^\pm,\xi)$ of \eqref{eq:dispersion-Lag} are
\begin{equation*}
    \lambda_{1,2}^\pm=-\frac{\xi^2\mu(V^\pm)}{2V^\pm}+\xi si\pm\frac12\sqrt{\widetilde{\Delta}},
\end{equation*}
and we have the two following cases:
\begin{itemize}
    \item[(a)] If $\widetilde{\Delta}\leq0$ then 
    \begin{equation*}
        \Re(\lambda_{1,2}^\pm)=-\frac{\xi^2\mu(V^\pm)}{2V^\pm}\leq0, \qquad \qquad \forall\,\xi\in\R.
    \end{equation*}
    \item[(b)] If $\widetilde{\Delta}>0$ then trivially $\Re(\lambda_2^\pm)\leq 0$  for any $\xi\in\R$, while
    \begin{equation*}
        \Re(\lambda_1^\pm)=-\frac{\xi^2\mu(V^\pm)}{2V^\pm}+\frac12\sqrt{\frac{\xi^4\mu(V^\pm)^2}{(V^\pm)^2}+4p'(V^\pm)\xi^2-4\k(V^\pm)\xi^4},
    \end{equation*}
    and one has $\Re(\lambda_1^\pm)\leq0$, for any $\xi\in\R$ because
    \begin{equation*}
        4p'(V^\pm)\xi^2-4\k(V^\pm)\xi^4\leq0,
        \qquad \qquad \forall\,\xi\in\R,
    \end{equation*}
    in view of \eqref{cond_pressure-L} and the positivity of $\k$. 
\end{itemize}

In conclusion, we have proved that the curves $\xi \mapsto \lambda_{1,2}^\pm(\xi)$ solving \eqref{eq:dispersion-Lag} are in the closed left half-plane
$\left\{z\in\mathbb{C}\, :\, \Re z\leq0\right\}$, for any $V^\pm>0$.
Moreover, if $\xi\neq0$, then $\Re\lambda_{1,2}^\pm(\xi)<0$. Therefore, we define the following connected, open set in the complex plane,
\begin{equation*}
\Omega := \Big\{ \lambda \in \C \, : \, \Re \lambda > \max \{ \Re \lambda^+_{j}(\xi), \Re \lambda^-_j(\xi), \, j =1,2, \, \xi \in \R \} \Big\}.
\end{equation*}
From the above calculations it is clear that
\[
\{ \lambda \in \C \, : \, \Re \lambda > 0 \} \subseteq \Omega.
\]
Moreover, for each $\lambda \in \Omega$ the matrices $M_\pm(\lambda)$ are hyperbolic (because its eigenvalues have real part different from zero, $\theta \neq i \xi$, and thus there is no center eigenspace). For each $\lambda \in \C$, let us denote the stable and the unstable eigenspaces of $M_\pm(\lambda)$ as $\E^{s}_\pm(\lambda)$ and $\E^{u}_\pm(\lambda)$, respectively. Then we have the following result, a property known as \emph{consistent splitting} \cite{KaPro13,San02}.

\begin{proposition}[consistent splitting]
\label{propconssplit}
For all $\lambda \in \Omega$ there holds
\begin{equation}
\label{conssplit}
\dim \E^u_\pm(\lambda) = \dim \E^s_\pm(\lambda) = 2.
\end{equation}
\end{proposition}
\begin{proof}
By the connectedness of $\Omega$ and the continuity of $M_\pm(\lambda)$ as a function of $\lambda$, the dimensions of the stable/unstable eigenspaces are constant in $\Omega$. Thus, it suffices to compute these dimensions for $\lambda \in \R$, $\lambda \gg 1$, sufficiently large. In this fashion, the quartic equation $\det(M_\pm(\lambda)-\theta I)= 0$ becomes the following quartic equation with real coefficients (see \eqref{detMtheta}),
\begin{equation}
\label{realquartic}
a_4 \theta^4 + a_3 \theta^3 + a_2 \theta^2 + a_1 \theta + a_0 = 0, 
\end{equation}
where
\[
a_4 = 1, \quad a_3 = \frac{s \mu^\pm}{V^\pm \kappa^\pm}, \quad a_2 =  \frac{s^2+p'(V^\pm)}{\k^\pm} - \frac{\lambda\mu^\pm}{V^\pm\k^\pm}, \quad a_1 = - \frac{2 \lambda s}{\kappa^\pm}, \quad  a_0 = \frac{\lambda^2}{\kappa^\pm},
\]
and with $\mu^\pm = \mu(V^\pm), \kappa^\pm = \kappa(V^\pm)$. The discriminant of the quartic equation \eqref{realquartic} is of the form
\[
\widehat{\Delta} = \frac{256}{(\kappa^\pm)^3} \lambda^6 + O(\lambda^5) > 0,
\]
strictly positive for $\lambda \gg 1$. Moreover,
\[
\widehat{P} :=  8a_4a_2 - 3a_3^2 = \frac{8}{\kappa^\pm} \Big( s^2 + p'(V^\pm) - \lambda \frac{\mu^\pm}{V^\pm}\Big) - \frac{3 s^2 (\mu^\pm)^2}{(V^\pm)^2(\kappa^\pm)^2} < 0,
\]
which is negative for $\lambda \gg 1$, is the second order coefficient of the associated depressed quartic, and
\[
\begin{aligned}
\widehat{D} &:= 64 a_4^3 a_0 - 16 a_4^2 a_2^2 + 16 a_4 a_3^2 a_2 - 16 a_4^2 a_3 a_1 - 3 a_3^4 \\
&= \frac{16 \lambda^2}{\kappa^\pm} \Big( 4 - \frac{(\mu^\pm)^2}{\kappa^\pm(V^\pm)^2}\Big) + O(\lambda),
\end{aligned}
\]
has the sign of $4 - \frac{(\mu^\pm)^2}{\kappa^\pm(V^\pm)^2}$ for $\lambda \gg 1$. We have two cases (cf. \cite{Rees1922,Lzrd88}):
\begin{itemize}
\item[(a)] If $\widehat{\Delta} > 0$, $\widehat{P} < 0$ and $\widehat{D} < 0$ then all four roots of \eqref{realquartic} are real and distinct.
\item[(b)] If $\widehat{\Delta} > 0$, $\widehat{P} < 0$ and $\widehat{D} > 0$ then \eqref{realquartic} has two pairs of (non-real) complex conjugated roots.
\end{itemize}
Let us first examine case (a). We have two subcases:
\begin{itemize}
\item[(i)] The case of a 1-backward shock for which $s < 0$. In this case $a_4 = 1 > 0$, $a_3 < 0$, $a_2 < 0$, $a_1 > 0$ and $a_0 > 0$ in the regime $\lambda \in \R$, $\lambda \gg 1$. By the Descartes' rule of signs \cite{Wang-X04} the number $\hat{p}$ of positive real roots of \eqref{realquartic} is either $\hat{p} = 2$ or $\hat{p} < 2 + 2k$ with $k \geq 1$. The number $\hat{q}$ of negative roots is the number of positive roots of $a_4 \theta^4 - a_3 \theta^3 + a_2 \theta^2 - a_1 \theta + a_0 = 0$. By the same argument, since $a_4 > 0$, $-a_3 > 0$, $a_2 < 0$, $-a_1 < 0$ and $a_0 > 0$ we have that either $\hat{q} = 2$ or $\hat{q} < 2 + 2m$ with $m \geq 1$. Suppose that $\hat{p} < 2$. Then necessarily $\hat{p} = 0$ because it is less than the upper bound by an even number $2k$, with $k \geq 1$. But this implies that the four roots are real and negative, a contradiction with $\hat{q} \leq 2$. We conclude that $\hat{p} = \hat{q} = 2$ and \eqref{conssplit} follows.
\item[(ii)] The case of a 2-forward shock is analogous. Here $s > 0$ and therefore $a_4  > 0$, $a_3 > 0$, $a_2 < 0$, $a_1 < 0$ and $a_0 > 0$ for $\lambda \gg 1$. This yields $\hat{p} \leq 2$, the number of positive roots. Likewise, $a_4  > 0$, $-a_3 < 0$, $a_2 < 0$, $-a_1 > 0$ and $a_0 > 0$ implies that $\hat{q} \leq 2$. Once again by Descartes' rule of signs we obtain $\hat{p} = \hat{q} = 2$ and the conclusion follows.
\end{itemize}

We now consider case (b). Clearly, for $\lambda \gg 1$ (which implies $\lambda \in \Omega$) the roots of \eqref{realquartic} are not purely imaginary, $\theta \neq i \xi$, $\xi \in \R$. Moreover, the second order coefficient is always negative for $\lambda \gg 1$ sufficiently large, $a_2 < 0$, and the leading coefficient is $a_4 \equiv 1$. This implies that the roots of \eqref{realquartic} cannot be all located in the complex left half plane, $\{ z \in \C, \: \, \Re z < 0 \}$, because in that case the coefficients should all be positive. Indeed, if all roots have negative real part then the quartic polynomial can be recast as the product of linear factors of the form $\theta - \theta_0$ with $\theta_0 < 0$ (real roots) and quadratic factors of the form $(\theta - \alpha)^2 + \beta^2$ (complex conjugated roots, $\alpha \pm i \beta$ with $\alpha < 0$), yielding a quartic equation with positive coefficients only. Moreover, substituting $\theta \to - \theta$ the second order coefficient remains negative; hence, there are roots also in the left half plane. We conclude that equation \eqref{realquartic} has exactly two complex conjugate roots with positive real part and two complex conjugate roots with negative real part, yielding the result in both the 1-backward shock and the 2-forward shock cases.

The Proposition is now proved.
\end{proof}

As a direct consequence of the above computations we readily obtain the stability of the essential spectrum.

\begin{theorem}[stability of the essential spectrum]
\label{thmstabess}
The essential spectrum of the linearized operator around any viscous-dispersive shock profile is stable, independently of the shock strength and for all the choices of end states $V^\pm > 0$.
\end{theorem}
\begin{proof}
Suppose that the end states $(V^\pm,U^\pm)$ and the speed $s$ define a Lax shock for the $p$--system \eqref{eq:Euler-Lag} and that $(V,U)(y)$ is the corresponding viscous-dispersive shock profile. Let $\tcL : L^2 \times L^2 \to L^2 \times L^2$ be the linearized operator around the profile defined in \eqref{deflinop}. For any $\lambda \in \C$ let us define $i_\pm(\lambda) := \dim \E^u_\pm(\lambda)$, the dimension of the unstable eigenspace of $M_\pm(\lambda)$. We now invoke Theorem 3.1.11 in Kapitula and Promislow \cite{KaPro13} (which is based upon on Weyl's essential spectrum theorem) to obtain
\[
\ess(\tcL) = \{ \lambda \in \C \, : \, i_-(\lambda) \neq i_+(\lambda)\} \cup \{ \lambda \in \C \, : \, \dim \E^c_\pm(\lambda) \neq 0\},
\]
where $\E^c_\pm(\lambda)$ denotes the center eigenspace of $M_\pm(\lambda)$ for each $\lambda \in \C$.

Now let $\lambda \in \Omega$. From consistent splitting (Proposition \ref{propconssplit}) and hyperbolicity of $M_\pm(\lambda)$ we clearly obtain $i_-(\lambda) = i_+(\lambda)$ and $\dim \E_\pm^c(\lambda) = 0$, yielding $\Omega \subset \C \backslash \ess(\tcL)$. This implies, in turn, that $\ess(\tcL) \subset \C \backslash \Omega \subset \{ \lambda \in \C \, : \, \Re \lambda \leq 0 \}$. Moreover, since $\Re \lambda_{1,2}^\pm (\xi) < 0$ for all $\xi \neq 0$, we conclude that
\[
\ess(\tcL) \subset \{ \lambda \in \C \, : \, \Re \lambda < 0 \} \cup \{ 0 \},
\]
proving the result.
%\hole{(G) TO DO: Finish the proof.}
\end{proof}

\begin{remark}
It is to be observed that this result holds for any 1-backward or 2-forward shock, that is, for any triplet $(V^\pm,U^\pm, s)$ satisfying Rankine--Hugoniot and Lax entropy conditions, independently of the shock amplitude. The essential spectrum of linearized operator around the unique (up to translations) viscous-dispersive profile is therefore stable. Moreover, thanks to the calculation of the dispersive relation we can say more about its location. If we define the Fredholm borders as
\[
\sigma_F(\tcL) := \{ \lambda \in \C \, : \, d_\pm(i\xi, \lambda) = 0, \, \xi \in \R \},
\]
then it is known (see Theorem 3.1.13 in \cite{KaPro13}) that these Fredholm borders are the boundaries of the open regions on which the operator is Fredholm; hence, $\ess(\tcL)$ is located to the left of the union of the Fredholm borders with maximum real part. Since in our case these borders touch the origin, we observe accumulation of the essential spectrum near the origin (such as in the purely viscous shock case \cite{ZH98,GZ98}). 
\begin{figure}[t]
\begin{center}
\includegraphics[scale=.62, clip=true]{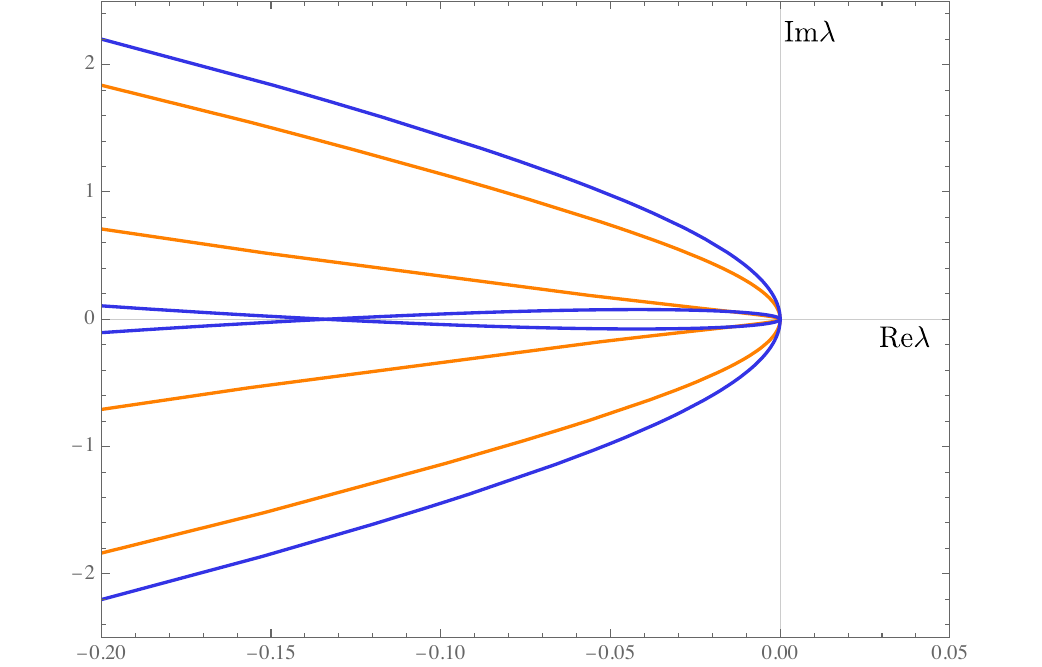}
\end{center}
\caption{\small{Fredholm borders, $\sigma_F(\tcL)$, as the curves $\lambda = \lambda_{1,2}^-(\xi)$ (in blue) and $\lambda = \lambda_{1,2}^+(\xi)$ (in orange) for $\xi \in \R$, which are solutions to the dispersion relation \eqref{eq:dispersion-Lag}. The essential spectrum of $\tcL$ is sharply bounded to the left. The accumulation of essential spectrum near the origin is manifest (color online).}}\label{figFredholm}
\end{figure}
Let us now illustrate the location of the essential spectrum by computing the Fredholm borders in the complex plane in the case of the adiabatic $\gamma$-gas pressure law with (nonlinear and $v$-dependent) viscosity and capillarity coefficients defined in \eqref{numexample}. Once again we select the end states to be $V^- = 1.5$ and $V^+ = 1$. Figure \ref{figFredholm} shows the complex roots $\lambda = \lambda_{1,2}^+(\xi)$ (in orange color) and $\lambda = \lambda_{1,2}^-(\xi)$ (in blue color) of the dispersion relation \eqref{eq:dispersion-Lag}. 
The essential spectrum is located to the left of these curves. Notice that any stable neighborhood of the origin contains elements of the essential spectrum. This fact is referred to as the absence of a spectral gap between the spectrum of $\tcL$ and the origin, making the nonlinear (asymptotic) stability analysis of the profiles more complicated than in the standard case.
%\hole{(H) TO DO: Finish description.}
\end{remark}

\section{Point spectral stability of small shocks}
\label{sec:point-L}

In this Section we examine the point spectrum of the linearized operator around any viscous-dispersive profile in the case when its amplitude is sufficiently small. Following Goodman \cite{Go86,Go91}, it is more convenient to recast the spectral problem in terms of an integrated operator. For that purpose, let us introduce the integrated variables
%
%\hole{(I) TO DO: Rewrite this section. Include the lemma of equivalence with the spectrum of the integrated operator.}
%The goal of this Section is to study the eigenvalue problem \eqref{eq:eigenvalue1-Lag}-\eqref{eq:eigenvalue2-Lag} and to prove the stability of the point spectrum, that is we will prove that if $\lambda\in\mathbb{C}$ is an eigenvalue of \eqref{eq:eigenvalue1-Lag}-\eqref{eq:eigenvalue2-Lag}, then $\Re \lambda<0$.
%
%
%
%
%
%Let us introduce the integrated variables
\begin{equation*}
    v(x):=\int_{-\infty}^x \tiv(y)\,dy, 
    \qquad u(x):=\int_{-\infty}^x \tiu(y)\,dy.
\end{equation*}
This transformation removes the zero eigenvalue without any further modification of the spectrum. Indeed, the original eigenvalue problem \eqref{eq:eigenvalue-L-Lag} can be expressed in terms of $v, u$ as 
\begin{align}
        \lambda v&=sv'+u', \label{eq:eigenvalue1-Lag}\\
        \lambda u&=su'-p'(V)v' + \cL_v\begin{pmatrix}v\\ u\end{pmatrix} + \cL_c v, \label{eq:eigenvalue2-Lag}
\end{align}
where $'$ stands for $d/dx$ and the operators $\cL_v$ and $\cL_c$ are defined as
\begin{equation*}
    \cL_v\begin{pmatrix}v\\ u\end{pmatrix}:=\frac{\mu(V)}{V}u''+\frac{\mu'(V)V-\mu(V)}{V^2}U'v',
\end{equation*}
and 
\begin{equation}\label{eq:cap}
    \cL_c v:=-\k(V)v'''-\k'(V)V''v'-\k'(V)V'v''-\frac12\k''(V)(V')^2v',
\end{equation}
respectively. Hence, we define the \emph{integrated operator} as
%As it is well known, the spectrum consists of two parts: the essential spectrum and the point spectrum. Let us first analyze the former.
%
%
%
%
%
%Let us denote by
\begin{equation*}
    \mathcal{L}\begin{pmatrix}v\\ u\end{pmatrix}:=\begin{pmatrix}
        sv'+u'\\
        su'+\displaystyle\left[-p'(V)+\frac{\mu'(V)V-\mu(V)}{V^2}U'\right]v'+\frac{\mu(V)}{V}u''+ \cL_c v
    \end{pmatrix},
\end{equation*}
\[
\cL :  L^2 \times L^2 \to L^2 \times L^2,
\]
with dense domain $\cD(\cL) = H^3 \times H^2 \subset L^2 \times L^2$. Once again, the integrated operator $\cL$ is a closed, densely defined operator acting on the energy space $L^2 \times L^2$. The relation between the point spectra of both operators is described in the following Proposition, whose proof is very similar (actually, it is the same, word by word) to that of Lemma 3.4 in \cite{FPZ22} and therefore we omit it.
\begin{proposition}
\label{propequiv}
The point spectrum of the original operator $\tcL$ is contained in the point spectrum of the integrated operator $\cL$, except for the eigenvalue zero. More precisely, 
\[
\ptsp(\tcL) \backslash \{0\} \subset \ptsp(\cL).
\]
\end{proposition}

In view of this result, it suffices to prove that $\ptsp(\cL) \subset \{ \lambda \in \C \, : \, \Re \lambda < 0 \}$ in order to obtain the point spectral stability of the original operator. The following proposition establishes sufficient conditions for that to happen. For simplicity we consider the case of a 1-backward shock with $s < 0$ and $V^- > V^+$ (the case of a 2-forward shock is analogous).
%
%$L_c$ is defined in \eqref{eq:cap} and let $\ptsp(\mathcal{A})$ be the point spectrum of $\mathcal{A}$.
\begin{proposition}\label{prop:energy-est}
Let $(V,U)$ be the solution of \eqref{eq:profiles-L}-\eqref{eq:profiles-end-Lag} given by Theorem \ref{theo:ex-L}. 
Moreover, let us assume that $V^->0$ is such that
\begin{equation}
\label{eq:imp-ass}
    -\k'(V^-)p'(V^-)+\k(V^-)p''(V^-)=:M>0.
\end{equation}
If the shock amplitude $\varepsilon=V^--V^+>0$ is sufficiently small, then the operator $\mathcal{L}$ is point spectrally stable, that is $$\ptsp(\mathcal{L})\subset\left\{\lambda\in\mathbb{C}\, : \, \Re\lambda<0\right\}.$$
\end{proposition}
\begin{remark}
The proof of Proposition \ref{prop:energy-est} extends the energy estimates performed by Humpherys \cite{Hu09} to analyze the system with constant viscosity and capillarity coefficients in Lagrangian coordinates, a case under which condition \eqref{eq:imp-ass} is automatically satisfied. A similar strategy has also been used to study a quantum hydrodynamic system with linear \cite{FPZ22} and nonlinear \cite{FPZ23} viscosities.
\end{remark}

Before proving Proposition \ref{prop:energy-est}, we need some auxiliary results.

\begin{lemma}\label{lem:f1}
Let $(V,U)$ be the solution of \eqref{eq:profiles-L}-\eqref{eq:profiles-end-Lag} given by Theorem \ref{theo:ex-L}. Define
\begin{equation}\label{eq:f1}
    f_1(x):=-p'(V(x))+\frac{\mu'(V(x))V(x)-\mu(V(x))}{V^2(x)}U'(x).
\end{equation}
If $V^->0$ and the amplitude $\e:=V^--V^+$ is sufficiently small, then there exist positive constants $C_1,C_2>0$ (independent on $\e$) such that
\begin{align}
    0<C_1\leq f_1(x)&\leq C_2,  \label{eq:f1-bounded}\\
    0<C_1|V'(x)|\leq f_1'(x)&\leq C_2|V'(x)|,  \label{eq:f1-der-bounded}\\
    |f_1''(x)|&\leq C\e|V'(x)|, \label{eq:f1-der2}
\end{align}
for any $x\in\R$.
\end{lemma}
\begin{proof}
The result is a direct consequence of Theorem \ref{theo:monotone-L}. 
First, since $U'=-sV'$ one has
\begin{equation*}
    f_1(x):=-p'(V(x))-s\frac{\mu'(V(x))V(x)-\mu(V(x))}{V^2(x)}V'(x).
\end{equation*}
From \eqref{eq:estprof-L} it follows that
\begin{equation*}
    f_1(x):=-p'(V(x))+\mathcal{O}(\e^2),
\end{equation*}
and the assumption on $p$ \eqref{cond_pressure-L} together with the monotonicity of $V$ imply
\begin{equation*}
    -p'(V(x))\geq -p'(V^+)>0, \qquad \mbox{ for } V^+>0.
\end{equation*}
Hence, if the amplitude $\e$ of the shock is sufficiently small, we end up with \eqref{eq:f1-bounded}.
By differentiating, we obtain
\begin{equation}\label{eq:f1-prime}
    f_1'=-p''(V)V'+\omega_1,
\end{equation}
where
\begin{equation*}
    \omega_1:=-s\frac{\mu'(V)V-\mu(V)}{V^2}V''
    -s\frac{\mu''(V)V^2-2\mu'(V)V+2\mu(V)}{V^3}(V')^2.
\end{equation*}
Thanks to \eqref{eq:estprof-L}-\eqref{eq:estprof-2-L}, we can state that there exists a constant $C>0$ (independent on $\e$) such that
\begin{equation}\label{eq:est-omega}
    |\omega_1|\leq C\e|V'|,
\end{equation}
if $\e$ is sufficiently small, and as a consequence, from \eqref{eq:f1-prime} we deduce \eqref{eq:f1-der-bounded} still in view of the monotonicity of $V$ and the convexity of the pressure $p$ in \eqref{cond_pressure-L}.
Finally, a direct computation gives
\begin{equation*}
    f_1''=-p''(V)V''-p'''(V)(V')^2+\omega_1',
\end{equation*}
and, by using \eqref{eq:estprof-L}-\eqref{eq:estprof-2-L}-\eqref{eq:estprof-3-L}, we infer
\begin{equation*}
    |\omega_1'|\leq C\e^2|V'|.
\end{equation*}
Therefore, by using again \eqref{eq:estprof-L}-\eqref{eq:estprof-2-L}, we end up with \eqref{eq:f1-der2} and the proof is complete.
\end{proof}
\begin{lemma}
Let $(V,U)$ be the solution of \eqref{eq:profiles-L}-\eqref{eq:profiles-end-Lag} given by Theorem \ref{theo:ex-L}. Define
\begin{equation}\label{eq:f2}
    f_2:=\left\{\frac{s}{2f_1(x)}-\left[\frac{\mu(V(x))}{2V(x)f_1(x)}\right]'\right\}',
\end{equation}
where $f_1$ is defined in \eqref{eq:f1}.
If $V^->0$ and the amplitude $\e:=V^--V^+$ is sufficiently small, then there exist positive constants $C_1,C_2>0$ (independent on $\e$) such that
\begin{equation}\label{eq:f2-bounded}
    0<C_1|V'(x)|\leq f_2(x)\leq C_2|V'(x)|,
\end{equation}
for any $x\in\R$.
\end{lemma}
\begin{proof}
A direct differentiation and \eqref{eq:f1-prime} give
\begin{equation*}
    f_2=\frac{sp''(V)V'}{2f_1^2}+\omega_2,
\end{equation*}
where 
\begin{equation*}
    \omega_2:=-\frac{s\omega_1}{2f_1^2}-\left[\frac{\mu(V)}{2Vf_1}\right]''.
\end{equation*}
By using \eqref{eq:f1-bounded}, \eqref{eq:f1-der-bounded}
-\eqref{eq:f1-der2} and \eqref{eq:est-omega},
we deduce
\begin{equation*}
    |\omega_2|\leq C\e|V'|,
\end{equation*}
and, as a consequence, since $sV'>0$ one has $f_2>0$. Moreover, thanks to \eqref{cond_pressure-L} and \eqref{eq:f1-bounded} we can conclude that $f_2$ satisfies \eqref{eq:f2-bounded}. 
\end{proof}
\begin{lemma}
Let $(V,U)$ be the solution of \eqref{eq:profiles-L}-\eqref{eq:profiles-end-Lag} given by Theorem \ref{theo:ex-L}. Define
\begin{equation}\label{eq:f3}
    f_3(x):=\frac{s}{2}\left[\frac{\k(V(x))}{f_1(x)}\right]'.
\end{equation}
Assume that $V^->0$, \eqref{eq:imp-ass} holds true and the amplitude $\e:=V^--V^+$ is sufficiently small. 
Then there exist positive constants $C_1,C_2>0$ (independent on $\e$) such that
\begin{equation}\label{eq:f3-bounded}
    0<C_1|V'(x)|\leq f_3(x)\leq C_2|V'(x)|,
\end{equation}
for any $x\in\R$.
\end{lemma}
\begin{proof}
Let $f_3:=f_3(x)$. A direct computation gives
$$f_3=\frac{s}{2}\left[\frac{\k'(V)V'f_1-\k(V)f_1'}{f_1^2}\right].$$
Substituting the definition \eqref{eq:f1} of $f_1$ and \eqref{eq:f1-prime}, we deduce
$$f_3=\frac{sV'}{2f_1^2}\left[-\k'(V)p'(V)+\k(V)p''(V)\right]+\omega_2,$$
where 
\begin{equation*}
    \omega_2:=\frac{sV'}{2f_1^2}\left[\frac{\mu'(V)V-\mu(V)}{V^2}\right]U'\k'(V)-\frac{s}{2f_1^2}\k(V)\omega_1.
    \end{equation*}
Thanks to \eqref{eq:est-omega} and the equality $U'=-sV'$, we get
\begin{equation*}
    |\omega_2| \leq C\e|V'|.
\end{equation*}
Moreover, if $\e$ is sufficiently small,  $\displaystyle\frac{sV'(x)}{2f_1^2(x)}>0$ and 
\begin{equation*}
    0<\frac{M}{2}\leq-\k'(V(x))p'(V(x))+\k(V(x))p''(V(x))\leq 2M,
\end{equation*}
for any $x\in\R$.
Indeed, the first inequality follows from Theorem \ref{theo:monotone-L} and \eqref{eq:f1-bounded}, while
the second one is a consequence of assumption \eqref{eq:imp-ass} and the fact that $V(x)\in[V^+,V^-]$, for any $x\in\R$,
with $V^--V^+=\e$ sufficiently small. 
By combining all the estimates, we conclude that $f_3$ satisfies \eqref{eq:f3-bounded}.
\end{proof}
We have now all the tools to prove Proposition \ref{prop:energy-est}.
\begin{proof}[Proof of Proposition \ref{prop:energy-est}]
By contradiction, let us assume that there exists $\lambda\in \ptsp(\cL)$ with $\Re\lambda\geq0$.
Multiply equation \eqref{eq:eigenvalue2-Lag} by $u^*/f_1$,
where $f_1$ is the function defined in \eqref{eq:f1}
and integrate over $\mathbb{R}$ to obtain
\begin{equation}\label{eq:energy-est-Lag}
    \lambda\int_{\mathbb{R}}\frac{|u|^2}{f_1}\,dx=s\int_{\mathbb{R}}\frac{u^*u'}{f_1}\,dx+\int_{\mathbb{R}}u^*v'\,dx+\int_{\mathbb{R}}\frac{\mu(V)}{Vf_1}u^*u''\,dx+\int_{\mathbb{R}}\frac{u^*L_c v}{f_1}\,dx.
\end{equation}
Integrating by parts and using the first equation \eqref{eq:eigenvalue1-Lag}, we can rewrite the second term in the right hand side of \eqref{eq:energy-est-Lag} as
\begin{equation*}
      \int_{\mathbb{R}}u^*v'\,dx=-\int_{\mathbb{R}}(u')^*v\,dx=-\int_{\mathbb{R}}(\lambda v-sv')^*v\,dx=-\lambda^*\int_{\mathbb{R}}|v|^2\,dx+s\int_{\mathbb{R}}(v')^*v\,dx.
\end{equation*}
By taking the real part of the latter equality, we infer
\begin{equation*}
    \Re\int_{\mathbb{R}}u^*v'\,dx=-(\Re\lambda)\int_{\mathbb{R}}|v|^2\,dx+\frac{s}2\int_{\mathbb{R}}\frac{d}{dx}(|v|^2)\,dx=-(\Re\lambda)\int_{\mathbb{R}}|v|^2\,dx,
\end{equation*}
and, as a consequence, \eqref{eq:energy-est-Lag} becomes
\begin{equation}\label{eq:energy-est-Lag2}
    \begin{aligned}
        (\Re\lambda)\int_{\mathbb{R}}\left[\frac{|u|^2}{f_1}+|v|^2\right]\,dx=&s\Re\int_{\mathbb{R}}\frac{u^*u'}{f_1}\,dx+\Re\int_{\mathbb{R}}\frac{\mu(V)}{Vf_1}u^*u''\,dx\\
        &\quad+\Re\int_{\mathbb{R}}\frac{u^*L_c v}{f_1}\,dx.
    \end{aligned}
\end{equation}
Similarly, one has
\begin{equation}\label{eq:u-u'}
    \Re\int_{\mathbb{R}}\frac{u^*u'}{f_1}\,dx=\frac12\int_{\mathbb{R}}\frac{1}{f_1}\frac{d}{dx}(|u|^2)\,dx=-\frac12\int_{\mathbb{R}}\left(\frac{1}{f_1}\right)'|u|^2\,dx
\end{equation}
and 
\begin{align}
    \Re\int_{\mathbb{R}}\frac{\mu(V)}{Vf_1}u^*u''\,dx&=-\int_{\mathbb{R}}\frac{\mu(V)}{Vf_1}|u'|^2\,dx-\Re\int_{\mathbb{R}}\left[\frac{\mu(V)}{Vf_1}\right]'u^*u'\,dx\notag\\
    &=-\int_{\mathbb{R}}\frac{\mu(V)}{Vf_1}|u'|^2\,dx+\int_{\mathbb{R}}\left[\frac{\mu(V)}{2Vf_1}\right]''|u|^2\,dx.\label{eq:u-u''}
\end{align}
Hence, thanks to \eqref{eq:u-u'}-\eqref{eq:u-u''} the relation \eqref{eq:energy-est-Lag2} becomes
\begin{equation}\label{eq:energy-est-Lag3}
        (\Re\lambda)\int_{\mathbb{R}}\left[\frac{|u|^2}{f_1}+|v|^2\right]\,dx+\int_{\mathbb{R}}\frac{\mu(V)}{Vf_1}|u'|^2\,dx+\int_{\mathbb{R}}f_2|u|^2\,dx=\Re\int_{\mathbb{R}}\frac{u^*L_c v}{f_1}\,dx,
\end{equation}
where $f_2$ is the function defined in \eqref{eq:f2}.
It remains to analyze the right hand side  of \eqref{eq:energy-est-Lag3}: the definition \eqref{eq:cap} implies
\begin{align*}
    \Re\int_{\mathbb{R}}\frac{u^*L_c v}{f_1}\,dx=&-\Re\int_{\mathbb{R}}\frac{\k(V)}{f_1}u^*v'''\,dx-\Re\int_{\mathbb{R}}\frac{\k'(V)V''}{f_1}u^*v'\,dx\\
    &\quad-\Re\int_{\mathbb{R}}\frac{\k'(V)V'}{f_1}u^*v''\,dx-\Re\int_{\mathbb{R}}\frac{\k''(V)(V')^2}{2f_1}u^*v'\,dx.
\end{align*}
Integration by parts gives
\begin{align*}
    \Re\int_{\mathbb{R}}\frac{u^*L_c v}{f_1}\,dx&=\Re\int_{\mathbb{R}}\frac{\k(V)}{f_1}(u')^*v''\,dx+\Re\int_{\mathbb{R}}\left[\frac{\k(V)}{f_1}\right]'u^*v''\,dx\\
    &\quad-\Re\int_{\mathbb{R}}\left[\frac{\k'(V)V''}{f_1}+\frac{\k''(V)(V')^2}{2f_1}\right]u^*v'\,dx\\
    &\quad+\Re\int_{\mathbb{R}}\frac{\k'(V)V'}{f_1}(u')^*v'\,dx+\Re\int_{\mathbb{R}}\left[\frac{\k'(V)V'}{f_1}\right]'u^*v'\,dx\\
    &=-\Re\int_{\mathbb{R}}\frac{\k(V)}{f_1}(u'')^*v'\,dx-2\Re\int_{\mathbb{R}}\left[\frac{\k(V)}{f_1}\right]'(u')^*v'\,dx\\
    &\quad+\Re\int_{\mathbb{R}}\frac{\k'(V)V'}{f_1}(u')^*v'\,dx+\Re\int_{\mathbb{R}}g_1u^*v'\,dx, 
\end{align*}
where
\begin{align*}
    g_1&:=-\left[\frac{\k(V)}{f_1}\right]''+\frac{\k''(V)(V')^2}{2f_1}+\k'(V)V'\left(\frac{1}{f_1}\right)'\\
    &=-\frac{\k'(V)V''}{f_1}-\k(V)\left(\frac{1}{f_1}\right)''-\frac{\k''(V)(V')^2}{2f_1}-\k'(V)V'\left(\frac{1}{f_1}\right)'.
\end{align*}
From \eqref{eq:eigenvalue1-Lag}, it follows that
\begin{equation*}
u''=\lambda v'-sv''
\end{equation*}
and, as a consequence,
\begin{align*}
    \Re\int_{\mathbb{R}}\frac{u^*L_c v}{f_1}\,dx&=- (\Re \lambda)\int_{\mathbb{R}}\frac{\k(V)}{f_1}|v'|^2\,dx-\frac{s}{2}\int_{\mathbb{R}}\left[\frac{\k(V)}{f_1}\right]'|v'|^2\,dx\\
    &\quad +\Re\int_{\mathbb{R}}g_1u^*v'\,dx-\Re\int_{\mathbb{R}}g_2(u')^*v'\,dx,
\end{align*}
where 
\begin{equation*}
    g_2:=\frac{\k'(V)V'}{f_1}+2\k(V)\left(\frac{1}{f_1}\right)'.
\end{equation*}
By substituting into \eqref{eq:energy-est-Lag3}, we end up with
\begin{align*}
        (\Re\lambda)\int_{\mathbb{R}}&\left[\frac{|u|^2}{f_1}+|v|^2+\frac{\k(V)}{f_1}|v'|^2\right]\,dx+\int_{\mathbb{R}}\frac{\mu(V)}{Vf_1}|u'|^2\,dx+\int_{\mathbb{R}}f_2|u|^2\,dx\\
        &\qquad \qquad+\int_{\mathbb{R}}f_3|v'|^2\,dx=\Re\int_{\mathbb{R}}g_1u^*v'\,dx-\Re\int_{\mathbb{R}}g_2(u')^*v'\,dx,
\end{align*}
where $f_3$ is the function defined in \eqref{eq:f3}.
The left hand side of the latter equality can be bounded from below by using \eqref{eq:f1-bounded}, \eqref{eq:f2-bounded} and \eqref{eq:f3-bounded}, while we can apply Young inequality to the right hand side; thus, there exists a constant $C>0$, independent on $\e$, such that
\begin{equation}\label{eq:nergy-est-Lag4}
    \begin{aligned}
         (\Re\lambda)\int_{\mathbb{R}}&\left[\frac{|u|^2}{f_1}+|v|^2+\frac{\k(V)}{f_1}|v'|^2\right]\,dx+C\int_{\mathbb{R}}|u'|^2\,dx+C\int_{\mathbb{R}}|V'||u|^2\,dx\\
        &\qquad +C\int_{\mathbb{R}}|V'||v'|^2\,dx\leq\frac12\int_{\mathbb{R}}|g_1||u|^2\,dx+\frac12\int_{\mathbb{R}}|g_1||v'|^2\,dx\\
        &\qquad \qquad \qquad \qquad\qquad\qquad+\eta_1\int_{\mathbb{R}}|u'|^2\,dx+\eta_2\int_\mathbb{R}|g_2|^2|v'|^2\,dx,
    \end{aligned}
\end{equation}
for some positive constants $\eta_1,\eta_2$ to be appropriately chosen later.
By using the definitions of $g_1$ and $g_2$, Theorem \ref{theo:monotone-L} and Lemma \ref{lem:f1}, we end up with
\begin{equation*}
    |g_1|\leq c\e|V'|, \qquad \qquad |g_2|\leq c|V'|,
\end{equation*}
for come positive constant $c$, which does not depend on $\e$. 
Hence, the energy estimate \eqref{eq:nergy-est-Lag4} becomes
\begin{equation*}
    \begin{aligned}
         (\Re\lambda)\int_{\mathbb{R}}\left[\frac{|u|^2}{f_1}+|v|^2+\frac{\k(V)}{f_1}|v'|^2\right]\,dx&+(C-\eta_1)\int_{\mathbb{R}}|u'|^2\,dx\\
         &+\left(C-\frac{c\e}{2}\right)\int_{\mathbb{R}}|V'||u|^2\,dx\\
        & +\left(C-\frac{c\e}{2}-\eta_2c_1\e^2\right)\int_{\mathbb{R}}|V'||v'|^2\,dx\leq0.
    \end{aligned}
\end{equation*}
In conclusion, by choosing $\e,\eta_1>0$ sufficiently small, we obtain a contradiction, namely, $\Re\lambda$ cannot be positive and the proof is complete.
\end{proof}
\begin{remark}\label{rem:extraass}
As we have already pointed out above, this spectral stability result under condition \eqref{eq:imp-ass} generalizes the result by Humpherys in \cite{Hu09}, where the viscosity and capillarity coefficients are constant, and therefore the condition is trivially satisfied in view of the positivity of (the constant) $\k$ and the convexity of $p$; surprisingly, our condition does not depend on the magnitude of the viscosity coefficient $\mu$. Moreover, in order to inspect and to better understand the meaning of this assumption, let us recast it in the case of coefficients obeying power laws. Recalling \eqref{eq:relationlagrangian}, the well-known $\gamma$-law pressure law becomes 
$p(v) = v^{-\gamma}$,  $\gamma\geq1$,
and
$\k(v) = v^{-\beta-5}$,
assuming a power law of the form $\tilde{\k}(\rho) = \rho^\beta$ for the capillarity coefficient in Eulerian coordinates, which depends on the density $\rho=1/v$. 
Hence, condition  \eqref{eq:imp-ass}  rewrites
\begin{equation*}
(-\gamma(\beta+5) +\gamma(\gamma+1))[(V^-)]^{-\gamma-\beta-7}>0,
\end{equation*}
that is,
\begin{equation*}
\gamma >\beta+4,
\end{equation*}
being $\gamma$ in particular positive.
As a matter of fact, this relation requires high powers for the adiabatic exponent, with respect to the one of  the capillarity coefficient $\tilde{\k}(\rho) = \rho^\beta$; for instance, in the quantum hydrodynamic case, we have $\beta=-1$ and therefore the pressure should verify a $\gamma$--law with $\gamma>3$, outside of the usual physical meaning range. Clearly,  more singular capillarity coefficients, i.e.\  $\beta<-3$, allows for any values of $\gamma\geq1$; the latter case includes the one with constant capillarity, obtained for  $\beta=-5$.
\end{remark}
%
%\hole{(J) TO DO: Add section of concluding remarks and future work.}

\section*{Acknowledgements}

The authors are grateful to Delyan Zhelyazov for many interesting discussions. 
The work of R. Folino was partially supported by DGAPA-UNAM, program PAPIIT, grant IN-103425.
The work of R. G. Plaza was partially supported by SECIHTI, Mexico, grant CF-2023-G-122.

%%% Beginning of the Appendix
\appendix
\section{Results in Eulerian coordinates}
\label{Eulerian}

In this Appendix we briefly describe the results concerning viscous-dispersive shock profiles in the case of the Navier--Stokes--Korteweg system in Eulerian coordinates. 
 Since the strategy and arguments of the proofs are the same of  those presented in previous sections, we gloss over many details and focus on the main differences with the Lagrangian formulation.

 First of all, let us recall the Navier--Stokes--Korteweg system in Eulerian coordinates:
\begin{equation}
\label{vK}
\begin{cases}
\rho_t + m_x = 0,\\
m_t + \left ( \frac{m^2}{\rho} + p(\rho) \right)_x = \left ( \mu(\rho) \left ( \frac{m}{\rho} \right )_x + K(\rho, \rho_x) \right )_x.
\end{cases}
\end{equation}
%\begin{equation}\label{eq:NSK}
%\begin{cases}
%\rho_t +  (\rho u  )_x = 0,\\
 %(\rho u  )_t + \left(\rho u^2 + p(\rho)\right )_x = \left (\mu(\rho) u_x\right )_x + K(\rho, \rho_x)_x.
%\end{cases}
%\end{equation}
In \eqref{vK}, 
we recall that $\rho:=\rho(x,t) > 0$ is the density, $u:=u(x,t)$ is the velocity and $m:= \rho u$ is the momentum, $\mu(\rho)>0$ is a (positive) viscosity coefficient depending on the density, which is assumed to be sufficiently smooth, and
\begin{equation*}
K(\rho,\rho_x) = \rho \k(\rho) \rho_{xx} + \frac{1}{2}(\rho \k'(\rho) - \k(\rho)) \rho_x^2
\end{equation*}
is the Korteweg capillarity tensor, where $\k(\rho)>0$ is a (positive) sufficiently smooth function. Finally, for the sufficiently smooth pressure function $p(\rho)$,  we assume 
\begin{equation}
\label{cond_pressure}
p'(\rho) > 0, \quad p''(\rho) > 0, \qquad \mbox{ for }\rho > 0.
\end{equation}
Traveling wave profiles are solutions of \eqref{vK} of the form
\begin{equation*}
    \rho(x,t) = R(x - st), \qquad m(x,t) = J(x - st), 
\end{equation*}
where $s \in \R$ is the speed of the traveling wave, with prescribed end states at $\pm\infty$: 
\begin{equation}\label{eq:profiles-end}
    R^\pm = \lim_{y \rightarrow \pm \infty} R(y), \qquad  J^\pm = \lim_{y \rightarrow \pm \infty} J(y),
\end{equation}
where we introduced the parameter along the profile $y = x - s t$.
Moreover, the end states $(R^\pm,J^\pm)$ and the speed $s$ are assumed to define a compressive shock of the first or second family for the underlying Euler equation
\begin{equation}\label{Euler_system}
\begin{cases}
\rho_t + (\rho u)_x = 0,\\
( \rho u) _t + \left ({\rho u^2}+ p(\rho) \right )_x = 0.
\end{cases}
\end{equation}
Clearly, the profiles $R$, $J$ solve
\begin{equation}\label{eq:profiles}
\begin{cases}
    -s R' + J' = 0,\\
-s J' + \left ( \frac{J^2}{R} + p(R) \right ) ' = \left ( \mu(R) \left (\frac{J}{R} \right) ' + K (R, R') \right )',
\end{cases}
\end{equation}
where
\begin{equation*}
K(R,R') = R \k(R) R'' + \frac{1}{2} (R \kappa'(R) - \k(R)) (R')^2.
\end{equation*}
The existence of viscous--dispersive shock profiles states as follows.
\begin{theorem}\label{theo:ex}
    Assume  the end states $(R^\pm,J^\pm)$ and the speed $s$ define a compressive $1$--shock, respectively $2$--shock, for the Euler equation \eqref{Euler_system}. Suppose that the viscosity coefficient $\mu$ and $\kappa$ are smooth, positive functions of $\rho>0$ and assume the pressure verifies \eqref{cond_pressure}. Then, there exists a unique (up to translation) traveling profile solving \eqref{eq:profiles-end} and \eqref{eq:profiles}. 
    Moreover, this travelling wave is oscillating whenever
    \begin{equation}
\label{osccond-E}
0 < \eta(R^-) < \frac{2 \sqrt{-{(R^-)}^5\widehat{f}'(R^-)}}{|A| }, \ \hbox{respectively}\ 
0 < \eta(R^+) < \frac{2 \sqrt{-{(R^+)}^5\widehat{f}'(R^+)}}{|A| }.
\end{equation}
\end{theorem}

The strategy of the proof of this result is the same used in the previous sections and we briefly describe it here below for the case of a compressive $1$--shock. Using Rankine--Hugoniot conditions, \eqref{eq:profiles} can be rewritten as
\begin{equation}\label{eq:R-Q}
    \begin{cases}
        R' = Q,\\
        Q' = \displaystyle\frac{ \widehat{f}(R)}{R \k(R)} - \frac{A \mu(R)}{R^3 \k(R)} Q - \frac{1}{2} \left ( \frac{\k'(R)}{\kappa(R)} - \frac{1}{R} \right ) Q^2,
    \end{cases}     
\end{equation}
where 
\begin{equation*}\label{eq:f(R)}
     \widehat{f}(R) := p(R) - p(R^-) +\frac{A^2}{R} - \frac{A^2}{R^-}
\end{equation*}
and
\begin{equation*}\label{def_A}
     A := sR^+-J^+=sR^--J^-.
\end{equation*}
The auxiliary system in this case is then given by
\begin{equation}\label{eq:R-Q-reduced}
    \begin{cases}
        R' = Q,\\
        Q' = \displaystyle\frac{ \widehat{f}(R)}{R \k(R)} - \frac{1}{2} \left ( \frac{\k'(R)}{\kappa(R)} - \frac{1}{R} \right ) Q^2,
    \end{cases}     
\end{equation}
obtained from \eqref{eq:R-Q} after removing the ``viscous'' term $- (A \mu(R)/R^3 \k(R)) Q$. 
Then, the constant $A$ is negative for a compressive $1$--shock, while the convex function $\widehat{f}$ plays the role of the convex function $f$ defined in the case of Lagrangian coordinates, and   it verifies
\begin{equation*}%\label{eq:sign-f}
    \widehat{f}(R)<0,\ \hbox{for any}\ R\in(R^-,R^+);\qquad  \widehat{f}'(R^-)<0, \  \widehat{f}'(R^+)>0.
\end{equation*}
As a consequence,  the arguments leading to the existence of the profiled described above can be easily carried out also in the present case \emph{mutatis mutandis}; see also Figure \ref{fig:comp-manifolds}.
\begin{figure}
\centering
\includegraphics{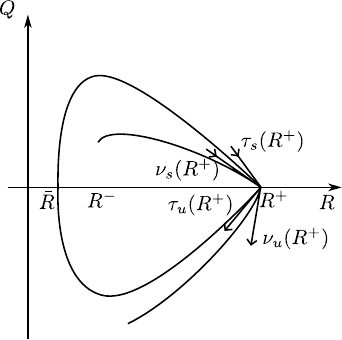}
\caption{Stable and unstable manifolds for \eqref{eq:R-Q} and \eqref{eq:R-Q-reduced} at $(R^+,0)$.}
\label{fig:comp-manifolds}
\end{figure}
Finally, the condition in \eqref{osccond-E} is equivalent to the fact that the
the discriminant evaluated at the equilibrium point $(R^-,0)$ inside the  (negatively for compressive $1$--shocks) invariant  region for \eqref{eq:R-Q}  is negative, so that 
this point is a stable focus, which implies in particular the profile is oscillating.

The analysis of monotonicity of the profiles for small shocks is  valid also for the model in Eulerian coordinates, and the corresponding results can be  proved  following the same lines of Section \ref{sec:properties-small-profiles-L}; we omit all details here.
\begin{theorem}\label{theo:monotone}
    Assume  the end states $(R^\pm,J^\pm)$ and the speed $s$ define a  $1$--shock (resp.\ $2$--shock) for the Euler equation \eqref{Euler_system} and assume its amplitude is sufficiently small: $0< \varepsilon = R^+ - R^- \ll 1$ (resp.\ 
    $0< \varepsilon = R^- - R^+ \ll 1$).
    Then,  the profile $R$ constructed in Theorem \ref{theo:ex} is monotone increasing (resp.\ decreasing) and it satisfies the following properties for any $y\in\mathbb{R}$:
    \begin{align*}
    &0<R'(y)<C\e^2,\ \hbox{resp.\ } -c\e^2<R'(y)<0,\\
    & |R''(y)|\leq C\e R'(y),
\end{align*}
where $C$ is a positive constant independent from $\varepsilon$.
\end{theorem}

Finally, concerning  stability properties of the spectrum of linearized operator around the profile, we emphasize that the analysis of the dispersion relation and its consequences in terms of the stability of essential spectrum by means of the study of Fredholm curves done in Section \ref{sec:essential-L} can be carried out  without significant modifications in full generality also in Eulerian coordinates. 
On the other hand, at the present stage, the validity of point spectral stability in this case  is not  proved even for small shocks. For this reasons, we consider the  results  concerning spectral properties of viscous--dispersive shocks in Eulerian coordinates  presently available too partial, and therefore we skip all further details in this context.

%%% End of the Appendix

%% To compile in my computer
\bibliography{bibliography}

\def\cprime{$'\!\!$}\def\cprime{$'\!\!$}\def\cprime{$'\!\!$}
\begin{thebibliography}{10}

\bibitem{BrDjL03}
{\sc D.~Bresch, B.~Desjardins, and C.-K. Lin}, {\em On some compressible fluid
  models: {K}orteweg, lubrication, and shallow water systems}, Commun. Partial
  Differ. Equ. \textbf{28} (2003), no.~3-4, pp.~843--868.

\bibitem{BrGL-V19}
{\sc D.~Bresch, M.~Gisclon, and I.~Lacroix-Violet}, {\em On
  {N}avier-{S}tokes-{K}orteweg and {E}uler-{K}orteweg systems: application to
  quantum fluids models}, Arch. Ration. Mech. Anal. \textbf{233} (2019), no.~3,
  pp.~975--1025.

\bibitem{CCD15}
{\sc Z.~Chen, X.~Chai, B.~Dong, and H.~Zhao}, {\em Global classical solutions
  to the one-dimensional compressible fluid models of {K}orteweg type with
  large initial data}, J. Differ. Equ. \textbf{259} (2015), no.~8,
  pp.~4376--4411.

\bibitem{CHZ15}
{\sc Z.~Chen, L.~He, and H.~Zhao}, {\em Nonlinear stability of traveling wave
  solutions for the compressible fluid models of {K}orteweg type}, J. Math.
  Anal. Appl. \textbf{422} (2015), no.~2, pp.~1213--1234.

\bibitem{CLS19}
{\sc Z.~Chen, Y.~Li, and M.~Sheng}, {\em Asymptotic stability of viscous shock
  profiles for the 1{D} compressible {N}avier-{S}tokes-{K}orteweg system with
  boundary effect}, Dyn. Partial Differ. Equ. \textbf{16} (2019), no.~3,
  pp.~225--251.

\bibitem{ChZha14}
{\sc Z.~Chen and H.~Zhao}, {\em Existence and nonlinear stability of stationary
  solutions to the full compressible {N}avier-{S}tokes-{K}orteweg system}, J.
  Math. Pures Appl. (9) \textbf{101} (2014), no.~3, pp.~330--371.

\bibitem{Da4e}
{\sc C.~M. Dafermos}, {\em Hyperbolic conservation laws in continuum physics},
  vol.~325 of Grundlehren der Mathematischen Wissenschaften, Springer-Verlag,
  Berlin, fourth~ed., 2016.

\bibitem{DD01}
{\sc R.~Danchin and B.~Desjardins}, {\em Existence of solutions for
  compressible fluid models of {K}orteweg type}, Ann. Inst. H. Poincar\'e Anal.
  Non Lin\'eaire \textbf{18} (2001), no.~1, pp.~97--133.

\bibitem{DS85}
{\sc J.~E. Dunn and J.~Serrin}, {\em On the thermomechanics of interstitial
  working}, Arch. Ration. Mech. Anal. \textbf{88} (1985), no.~2, pp.~95--133.

\bibitem{FNP24}
{\sc R.~Folino, A.~Naumkina, and R.~G. Plaza}, {\em Instability of periodic
  waves for the {K}orteweg--de {V}ries--{B}urgers equation with monostable
  source}, Phys. D \textbf{467} (2024), no.~134234, pp.~1--10.

\bibitem{FPZ22}
{\sc R.~Folino, R.~G. Plaza, and D.~Zhelyazov}, {\em Spectral stability of
  small-amplitude dispersive shocks in quantum hydrodynamics with viscosity},
  Commun. Pure Appl. Anal. \textbf{21} (2022), no.~12, pp.~4019--4040.

\bibitem{FPZ23}
\leavevmode\vrule height 2pt depth -1.6pt width 23pt, {\em Spectral stability
  of weak dispersive shock profiles for quantum hydrodynamics with nonlinear
  viscosity}, J. Differ. Equ. \textbf{359} (2023), pp.~330--364.

\bibitem{FrKo17}
{\sc H.~Freist\"{u}hler and M.~Kotschote}, {\em Phase-field and {K}orteweg-type
  models for the time-dependent flow of compressible two-phase fluids}, Arch.
  Ration. Mech. Anal. \textbf{224} (2017), no.~1, pp.~1--20.

\bibitem{FrKo19}
\leavevmode\vrule height 2pt depth -1.6pt width 23pt, {\em Phase-field
  descriptions of two-phase compressible fluid flow: interstitial working and a
  reduction to {K}orteweg theory}, Quart. Appl. Math. \textbf{77} (2019),
  no.~3, pp.~489--496.

\bibitem{GLZh20}
{\sc J.~Gao, Z.~Lyu, and Z.-a. Yao}, {\em Lower bound of decay rate for
  higher-order derivatives of solution to the compressible fluid models of
  {K}orteweg type}, Z. Angew. Math. Phys. \textbf{71} (2020), no.~4, p.~108.

\bibitem{GZ98}
{\sc R.~A. Gardner and K.~Zumbrun}, {\em The gap lemma and geometric criteria
  for instability of viscous shock profiles}, Comm. Pure Appl. Math.
  \textbf{51} (1998), no.~7, pp.~797--855.

\bibitem{GLT17}
{\sc J.~Giesselmann, C.~Lattanzio, and A.~E. Tzavaras}, {\em Relative energy
  for the {K}orteweg theory and related {H}amiltonian flows in gas dynamics},
  Arch. Ration. Mech. Anal. \textbf{223} (2017), no.~3, pp.~1427--1484.

\bibitem{Go86}
{\sc J.~Goodman}, {\em Nonlinear asymptotic stability of viscous shock profiles
  for conservation laws}, Arch. Ration. Mech. Anal. \textbf{95} (1986), no.~4,
  pp.~325--344.

\bibitem{Go91}
\leavevmode\vrule height 2pt depth -1.6pt width 23pt, {\em Remarks on the
  stability of viscous shock waves}, in Viscous profiles and numerical methods
  for shock waves (Raleigh, NC, 1990), M.~Shearer, ed., SIAM, Philadelphia, PA,
  1991, pp.~66--72.

\bibitem{HaSl83}
{\sc R.~Hagan and M.~Slemrod}, {\em The viscosity-capillarity criterion for
  shocks and phase transitions}, Arch. Ration. Mech. Anal. \textbf{83} (1983),
  no.~4, pp.~333--361.

\bibitem{HKKL25}
{\sc S.~Han, M.-J. Kang, J.~Kim, and H.~Lee}, {\em Long-time behavior towards
  viscous-dispersive shock for {N}avier-{S}tokes equations of {K}orteweg type},
  J. Differ. Equ. \textbf{426} (2025), pp.~317--387.

\bibitem{Hasp09}
{\sc B.~Haspot}, {\em Existence of strong solutions for nonisothermal
  {K}orteweg system}, Ann. Math. Blaise Pascal \textbf{16} (2009), no.~2,
  pp.~431--481.

\bibitem{Hasp11}
\leavevmode\vrule height 2pt depth -1.6pt width 23pt, {\em Existence of global
  weak solution for compressible fluid models of {K}orteweg type}, J. Math.
  Fluid Mech. \textbf{13} (2011), no.~2, pp.~223--249.

\bibitem{HaLi94}
{\sc H.~Hattori and D.~N. Li}, {\em Solutions for two-dimensional system for
  materials of {K}orteweg type}, SIAM J. Math. Anal. \textbf{25} (1994), no.~1,
  pp.~85--98.

\bibitem{HaLi96b}
\leavevmode\vrule height 2pt depth -1.6pt width 23pt, {\em The existence of
  global solutions to a fluid dynamic model for materials for {K}orteweg type},
  J. Partial Differ. Equ. \textbf{9} (1996), no.~4, pp.~323--342.

\bibitem{HaLi96a}
\leavevmode\vrule height 2pt depth -1.6pt width 23pt, {\em Global solutions of
  a high-dimensional system for {K}orteweg materials}, J. Math. Anal. Appl.
  \textbf{198} (1996), no.~1, pp.~84--97.

\bibitem{Hoe14}
{\sc M.~A. Hoefer}, {\em Shock waves in dispersive {E}ulerian fluids}, J.
  Nonlinear Sci. \textbf{24} (2014), no.~3, pp.~525--577.

\bibitem{HoAb07}
{\sc M.~A. Hoefer and M.~J. Ablowitz}, {\em Interactions of dispersive shock
  waves}, Phys. D \textbf{236} (2007), no.~1, pp.~44--64.

\bibitem{HZ00}
{\sc P.~Howard and K.~Zumbrun}, {\em Pointwise estimates and stability for
  dispersive-diffusive shock waves}, Arch. Ration. Mech. Anal. \textbf{155}
  (2000), no.~2, pp.~85--169.

\bibitem{HuTh02}
{\sc J.~Humpherys}, {\em Spectral energy methods and the stability of shock
  waves}, PhD thesis, Indiana University, 2002.

\bibitem{Hu09}
\leavevmode\vrule height 2pt depth -1.6pt width 23pt, {\em On the shock wave
  spectrum for isentropic gas dynamics with capillarity}, J. Differ. Equ.
  \textbf{246} (2009), no.~7, pp.~2938--2957.

\bibitem{J95}
{\sc C.~K. R.~T. Jones}, {\em Geometric singular perturbation theory}, in
  Dynamical systems ({M}ontecatini {T}erme, 1994), R.~Johnson, ed., vol.~1609
  of Lecture Notes in Math., Springer, Berlin, 1995, pp.~44--118.

\bibitem{KaPro13}
{\sc T.~Kapitula and K.~Promislow}, {\em Spectral and dynamical stability of
  nonlinear waves}, vol.~185 of Applied Mathematical Sciences, Springer-Verlag,
  New York, 2013.

\bibitem{Kat80}
{\sc T.~Kato}, {\em Perturbation Theory for Linear Operators}, Classics in
  Mathematics, Springer-{V}erlag, {N}ew {Y}ork, {S}econd~ed., 1980.

\bibitem{KhdPhD89}
{\sc M.~Khodja}, {\em Nonlinear stability of oscillatory traveling waves for
  some systems of hyperbolic conservation laws}, PhD thesis, University of
  Michigan, 1989.

\bibitem{Kortw1901}
{\sc D.~J. Korteweg}, {\em Sur la forme que prennent les \'{e}quations du
  mouvement des fluides si l'on tient compte des forces capillaires causées
  par des variations de densit\'{e} consid\'{e}rables mais continues et sur la
  th\'{e}orie de la capillarit\'{e} dans l'hypoth\`{e}se d'une variation
  continue de la densit\'{e}}, Arch. N\'{e}erl. Sci. Exactes Nat. Ser. II
  \textbf{6} (1901), pp.~1--24.

\bibitem{Kot08}
{\sc M.~Kotschote}, {\em Strong solutions for a compressible fluid model of
  {K}orteweg type}, Ann. Inst. H. Poincar\'e Anal. Non Lin\'eaire \textbf{25}
  (2008), no.~4, pp.~679--696.

\bibitem{Kot10}
\leavevmode\vrule height 2pt depth -1.6pt width 23pt, {\em Strong
  well-posedness for a {K}orteweg-type model for the dynamics of a compressible
  non-isothermal fluid}, J. Math. Fluid Mech. \textbf{12} (2010), no.~4,
  pp.~473--484.

\bibitem{LaZ21a}
{\sc C.~Lattanzio and D.~Zhelyazov}, {\em Traveling waves for quantum
  hydrodynamics with nonlinear viscosity}, J. Math. Anal. Appl. \textbf{493}
  (2021), no.~1, pp.~124503, 17.

\bibitem{Lzrd88}
{\sc D.~Lazard}, {\em Quantifier elimination: optimal solution for two
  classical examples}, J. Symbolic Comput. \textbf{5} (1988), no.~1-2,
  pp.~261--266.

\bibitem{PaW04}
{\sc J.~Pan and G.~Warnecke}, {\em Asymptotic stability of traveling waves for
  viscous conservation laws with dispersion}, Adv. Differ. Equ. \textbf{9}
  (2004), no.~9-10, pp.~1167--1184.

\bibitem{Pe85}
{\sc R.~L. Pego}, {\em Remarks on the stability of shock profiles for
  conservation laws with dissipation}, Trans. Amer. Math. Soc. \textbf{291}
  (1985), no.~1, pp.~353--361.

\bibitem{PlV22}
{\sc R.~G. Plaza and J.~M. Valdovinos}, {\em Dissipative structure of
  one-dimensional isothermal compressible fluids of {K}orteweg type}, J. Math.
  Anal. Appl. \textbf{514} (2022), no.~2, p.~Paper No. 126336.

\bibitem{Rees1922}
{\sc E.~L. Rees}, {\em Graphical discussion of the roots of a quartic
  equation}, Amer. Math. Monthly \textbf{29} (1922), no.~2, pp.~51--55.

\bibitem{San02}
{\sc B.~Sandstede}, {\em Stability of travelling waves}, in Handbook of
  dynamical systems, {V}ol. 2, B.~Fiedler, ed., North-Holland, Amsterdam, 2002,
  pp.~983--1055.

\bibitem{Sl83}
{\sc M.~Slemrod}, {\em Admissibility criteria for propagating phase boundaries
  in a van der {W}aals fluid}, Arch. Ration. Mech. Anal. \textbf{81} (1983),
  no.~4, pp.~301--315.

\bibitem{Sl84a}
\leavevmode\vrule height 2pt depth -1.6pt width 23pt, {\em Dynamic phase
  transitions in a van der {W}aals fluid}, J. Differ. Equ. \textbf{52} (1984),
  no.~1, pp.~1--23.

\bibitem{SmSh82}
{\sc J.~A. Smoller and R.~Shapiro}, {\em Dispersion and shock-wave structure},
  J. Differ. Equ. \textbf{44} (1982), no.~2, pp.~281--305.

\bibitem{TZh14}
{\sc Z.~Tan and R.~Zhang}, {\em Optimal decay rates of the compressible fluid
  models of {K}orteweg type}, Z. Angew. Math. Phys. \textbf{65} (2014), no.~2,
  pp.~279--300.

\bibitem{Teschl12}
{\sc G.~Teschl}, {\em Ordinary differential equations and dynamical systems},
  vol.~140 of Graduate Studies in Mathematics, American Mathematical Society,
  Providence, RI, 2012.

\bibitem{vdW1894}
{\sc J.~D. van~der Waals}, {\em Thermodynamische {T}heorie der
  {K}apillarit\"{a}t unter {V}oraussetzung stetiger {D}ichte\"{a}nderung}, Z.
  Phys. Chem. \textbf{13} (1894), no.~1, pp.~657--725.

\bibitem{Wang-X04}
{\sc X.~Wang}, {\em A simple proof of {D}escartes's rule of signs}, Am. Math.
  Mon. \textbf{111} (2004), no.~6, pp.~525--526.

\bibitem{WaTa11}
{\sc Y.~Wang and Z.~Tan}, {\em Optimal decay rates for the compressible fluid
  models of {K}orteweg type}, J. Math. Anal. Appl. \textbf{379} (2011), no.~1,
  pp.~256--271.

\bibitem{We10}
{\sc H.~Weyl}, {\em \"{U}ber gew\"ohnliche {D}ifferentialgleichungen mit
  {S}ingularit\"aten und die zugeh\"origen {E}ntwicklungen willk\"urlicher
  {F}unktionen}, Math. Ann. \textbf{68} (1910), no.~2, pp.~220--269.

\bibitem{ZLY16}
{\sc W.~Zhang, X.~Li, and Y.~Yong}, {\em Asymptotic stability of monotone
  increasing traveling wave solutions for viscous compressible fluid equations
  with capillarity term}, J. Math. Anal. Appl. \textbf{434} (2016), no.~1,
  pp.~401--412.

\bibitem{Z00}
{\sc K.~Zumbrun}, {\em Dynamical stability of phase transitions in the
  $p$-system with viscosity-capillarity}, SIAM J. Appl. Math. \textbf{60}
  (2000), no.~6, pp.~1913--1924.

\bibitem{ZH98}
{\sc K.~Zumbrun and P.~Howard}, {\em Pointwise semigroup methods and stability
  of viscous shock waves}, Indiana Univ. Math. J. \textbf{47} (1998), no.~3,
  pp.~741--871.

\end{thebibliography}

%This was the original style for the bibliography
%I chaged to \bibliographystyle{plain}, because with 

%%%%%%%%%%%%
\bibliographystyle{newstyle} 

%%%%%%%%%%%%\bibliographystyle{siam} 

%my computer does not compile
%\bibliographystyle{plain}

%the following style is for proofreading
%\bibliographystyle{amsalpha}
%\bibliographystyle{amsplain}

%% 
%% To compile anywhere
%%
% \bibliography{ref}
% \bibliographystyle{newstyle}

%\begin{thebibliography}{99}
%	
%\bibitem{Jones} 
%C. Jones, 
%\emph{Geometric singular perturbation theory}, 
%Dynamical systems (Montecatini Terme, 1994), 44-118, 
%Lecture Notes in Math., 1609, Springer, Berlin, 1995.
%	
%\bibitem{PhysD}
%C. Lattanzio, P. Marcati and D. Zhelyazov,
%\emph{Dispersive shocks in quantum hydrodynamics with viscosity},
%Phys. D {\bf402} (2020), 132222.
%
%\bibitem{AMC}
%C. Lattanzio, P. Marcati and D. Zhelyazov,
%\emph{Numerical investigations of dispersive shocks and spectral analysis for linearized quantum hydrodynamics},
%Appl. Math. Comput. {\bf385} (2020), 125450.
%
%\bibitem{JMAA}
%C. Lattanzio and D. Zhelyazov,
%\emph{Traveling waves for quantum hydrodynamics with nonlinear viscosity},
%J. Math. Anal. Appl. {\bf493} (2021), 124503.
%
%\bibitem{DelyanPreprint}
%D. Zhelyazov,
%\emph{Existence of standing and traveling waves in quantum hydrodynamics with viscosity}.
%Preprint 2021, 
%	
%\end{thebibliography}

\end{document}